\documentclass[10pt, reqno]{amsart}

\numberwithin{equation}{section}

\usepackage{amscd,amsmath}
\usepackage{mathrsfs}
\usepackage{amsfonts}
\usepackage{amssymb}
\usepackage{enumerate}
\usepackage{setspace}
\usepackage{color}

\usepackage[margin=1in]{geometry}

\usepackage[colorlinks=true, citecolor=blue]{hyperref}

\newtheorem{theorem}{Theorem}[section]
\newtheorem{lemma}{Lemma}[section]

\newtheorem{definition}{Definition}[section]

\newtheorem{remark}{Remark}[section]
\newtheorem{proposition}{Proposition}[section]

\newcommand{\eqn}{\begin{eqnarray}}
\newcommand{\een}{\end{eqnarray}}

\DeclareMathOperator{\dv}{div}
\DeclareMathOperator{\curl}{curl}

\makeatletter
 \newcommand{\norm}{\@ifstar{\@normb}{\@normi}}
 \newcommand{\@normb}[2]{\left\Vert{#1}\right\Vert_{#2}}
 \newcommand{\@normi}[2]{\Vert{#1}\Vert_{#2}}
 \makeatother

\usepackage{dirtytalk}

\usepackage[most]{tcolorbox}

\begin{document}

\title{On the double Beltrami states in Hall magnetohydrodynamics}

\author{Hantaek Bae}
\address[Hantaek Bae]{Department of Mathematical Sciences, Ulsan National Institute of Science and Technology (UNIST), Republic of Korea}
\email{hantaek@unist.ac.kr}

\author{Kyungkeun Kang}
\address[Kyungkeun Kang]{Department of Mathematics, Yonsei University, Republic of Korea}
\email{kkang@yonsei.ac.kr} 

\author{Jaeyong Shin}
\address[Jaeyong Shin]{Department of Mathematics, Yonsei University, Republic of Korea}
\email{sinjaey@yonsei.ac.kr}

\date{\today}

\begin{abstract}
In this paper, we investigate double Beltrami states in the Hall magnetohydrodynamic (Hall MHD) equations. Initially, we examine the double Beltrami states as a special class of steady solutions to the ideal Hall MHD equations, which are closely related to Beltrami flows in incompressible fluid dynamics. Specifically, we classify the double Beltrami states and show that they can be derived by using the variational method as energy minimizers, subject to the conservation of two helicities. We then extend our analysis to time-dependent double Beltrami states in the viscous and resistive Hall MHD equations, exploring their exact form and stability properties.
\end{abstract}

\maketitle

\vspace{-2ex}

\section{Introduction}
The Magnetohydrodynamics  (MHD in short) is the physical-mathematical framework that concerns the dynamics of magnetic fields in electrically conducting fluids. The electromagnetic field, the electric field $E$, and the magnetic field $B$  are governed by Maxwell's equations
\[
\begin{split}
\text{Amp\`ere's Law:} \ & \nabla \times  B=J, \\
\text{Faraday's Law:} \ & \nabla \times  E=-B_{t},\\
\text{Ohm's Law for resistive MHD:} \ &E+u\times B=\eta J, \\
\text{Gauss's Law for magnetism:} \ & \dv B=0, 
\end{split}
\]
where $u$ is the velocity field, $J$ is the current density, and $\eta$ is the resistivity (or magnetic diffusivity) of $B$. From these equations, we derive the induction equation: 
\[
B_{t}  + \nabla \times (B\times u) - \eta\Delta B=0.
\]
When a conducting fluid is permeated by $B$, the Lorentz force $J\times B$ drives motion in the fluid:
\[
u_{t}  +u\cdot \nabla u +\nabla p-\nu\Delta u=J\times B, \quad \dv u=0,
\]
where $\nu$ is the viscosity of the fluid and $p$ is the fluid pressure.

MHD provides a macroscopic description for fusion plasmas, the solar interior and its atmosphere, the Earth's magnetosphere and inner core, etc. However, {MHD} is deficient in many {respects}: for example, Ohm's law neglects the influence of the electric current in the Lorentz force. To resolve this issue, Lighthill derived a generalized Ohm's law \cite{Lighthill}:
\eqn \label{Generalized Ohm}
E+u\times B=\eta J+h\left(J\times B-\nabla p_{e}\right),
\een 
where $p_{e}$ is the electron pressure and $h$ is the ion skin depth. The  second term on the right-hand side of (\ref{Generalized Ohm})  is called the Hall term.  In terms of $(u,p,B)$, we obtain Hall MHD:
\begin{equation}\label{H MHD}
\begin{aligned}
&u_{t}  +\omega \times u+\nabla \overline{p}-\nu\Delta u=J\times B, \\
&B_{t}  +  \nabla \times (B\times u)+h\nabla \times \left(J\times B\right)- \eta\Delta B=0,  \\
&\dv u=\dv B=0,
\end{aligned}
\end{equation}
where $\omega=\nabla \times u$, $J=\nabla\times B$, and $\overline{p}=p+|u|^{2}/2$.

There is a large amount of {evidences} supporting that the Hall term plays a crucial role in describing many phenomena in plasma physics \cite{Balbus, Forbes, Homann, Mininni, Shalybkov, Shay, Wardle}. In particular, Hall MHD explains the magnetic reconnection on the Sun which is very important role in acceleration plasma by converting magnetic energy into bulk kinetic energy \cite{Forbes, Homann}. From the mathematical point of view, Hall MHD is derived from either two fluids model or kinetic models in a mathematically rigorous way  \cite{Acheritogaray}. Since then, Hall MHD  has been actively studied mathematically: see \cite{Bae Kang, Chae Degond Liu, Chae Lee, Chae Schonbek, Chae Wan Wu, Chae Weng, Chae Wolf 1, Chae Wolf 2, Dai 1, Dai 2, Dai 3, Danchin Tan 2, Danchin Tan 3, Jeong Oh, Kwak, Yamazaki} {and references therein}.

\vspace{1ex}

In this paper, we investigate double Beltrami states of the form \eqref{eq:HMHD-Beltrami} appearing in Hall MHD. The double Beltrami states have been studied in \cite{GIUAK, Kagan Mahajan, Mahajan Yoshida} and {are} known to meet the experimental results \cite{Priest Forbes, Schindler}. To introduce the double Beltrami states and explain the motivation behind our work, we briefly review Beltrami flows in incompressible fluid models in Section \ref{sec:1.1}, and force-free fields in MHD ($h=0$ in \eqref{H MHD}) in Section \ref{sec:1.2}.

\subsection{Beltrami flows} \label{sec:1.1}
We first review Beltrami flows. A fluid flow $u$ is called a Beltrami flow if there exists a scalar function $\lambda(x)\not\equiv0$ such that 
\begin{equation}\label{Beltrami flow}
\nabla \times  u(x)=\lambda(x)u(x),\quad\dv u(x)=0
\end{equation}
for all $x$ {in some domain}. Beltrami flows are special stationary solutions of the incompressible Euler equations. {A well-known family of Beltrami flows corresponding to a constant $\lambda$ is the Arnold-Beltrami-Childress (ABC) flows \cite{Arnold65, Childress, Dombre}: 
\eqn \label{ABC}
u(x)=\begin{pmatrix}
A\sin{\lambda x_{3}}+C\cos{\lambda x_{2}}\\
B\sin{\lambda x_{1}}+A\cos{\lambda x_{3}}\\
C\sin{\lambda x_{2}}+B\cos{\lambda x_{1}}
\end{pmatrix}.
\een
For a classification of Beltrami flows with constant $\lambda$ and more examples of Beltrami flows with non-constant $\lambda$, see \cite[Section 2.3.2]{Majda Bertozzi}. 
While Beltrami flows are smooth as long as $\lambda$ is smooth, we emphasize that they do not belong to $L^{2}(\mathbb{R}^{3})$. More precisely, there are no Beltrami flows in $L^{p}(\mathbb{R}^{3})$ for $p\in[2,3]$ \cite{Chae Constantin, Nadirashvili} even though there is a Beltrami flow  decaying as $u(x)\sim C/|x|$ as $|x|\rightarrow \infty$ \cite{Enciso Peralta-Salas}. It is conjectured by Arnold \cite{Arnold65} that Beltrami flows can exhibit the Lagrangian turbulence and chaotic behavior of trajectories to which there are several numerical evidences \cite{Arnold Khesin, Dombre, Henon}.
Interestingly, any incompressible flows can be represented as a superposition of Beltrami flows in periodic domain {$\mathbb{T}^{3}=[-\pi,\pi]^{3}$} \cite{Constantin Majda}. This idea is employed to construct dissipative continuous (incompressible) Euler flows in \cite{De Lellis Szekelyhidi}. For more properties of Beltrami flows, see \cite{Arnold Khesin, Chae Constantin, Chae Wolf, Enciso Peralta-Salas, Majda Bertozzi} and references therein.}

\vspace{1ex}

We also consider time-dependent Beltrami flows of the incompressible Navier-Stokes equations:
\eqn \label{NS}
u_{t}+u\cdot\nabla u+\nabla p-\nu \Delta u=0,\quad \dv u=0.
\een
Suppose that $\overline{u}=\overline{u}(t,x)$ satisfies \eqref{NS} and $\nabla \times \overline{u}=\lambda \overline{u}$ with a constant $\lambda$. Then, \eqref{NS} is reduced to 
\[
\overline{u}_t=\nu\Delta \overline{u}=-\nu\nabla\times\nabla\times \overline{u}=-\nu\lambda^2 \overline{u}.
\]
By choosing an initial Beltrami flow $\overline{u}_0=\overline{u}_0(x)$, the time-dependent Beltrami flow of  \eqref{NS}  is given by
\eqn \label{Trkalian flow}
\overline{u}(t,x)=e^{-\nu\lambda^{2}t}\overline{u}_0(x)
\een
which is also called a Trkalian flow  \cite{Lakhtakia, Trkal}.

\subsection{Force-free fields}\label{sec:1.2}
We now consider MHD:
\eqn \label{MHD}
\begin{aligned}
&u_{t}  +\omega \times u +\nabla \overline{p}-\nu \Delta u=J\times B,  \\
&B_{t}  +  \nabla \times (B\times u)-\eta \Delta B=0,  \\
&\dv u=\dv B=0.
\end{aligned}
\een 
 Formally, steady solutions of the ideal MHD ($\nu=\eta=0$ in (\ref{MHD})) are given by
\[
\omega \times u +\nabla \overline{p}=J\times B, \quad  \nabla \times (B\times u)=0.
\]
The magnetic energy $\mathcal{E}_{B}$ and the magnetic helicity $\mathcal{H}_{B}$ are defined by 
\[
\mathcal{E}_{B}(t)=\int_{\Omega} |B(t,x)|^{2}\,dx, \quad \mathcal{H}_{B}(t)=\int_{\Omega} (A\cdot  B)(t,x)\,dx, \quad \nabla\times A=B,
\]
where we choose $\Omega$ below. Woltjer \cite{Woltjer} applied the variational approach to show that a force-free field $B$ defined by 
\begin{equation}\label{eq:MHD-Beltrami-state}
\nabla\times B(x)=\lambda B(x),
\end{equation} 
with a constant $\lambda$, emerges as the magnetic energy $\mathcal{E}_B$ lowering process under the conservation of $\mathcal{H}_B$. As the magnetic energy reaches to a minimum level, the fluid comes to rest and the Lorentz force $J\times B$ vanishes. That is why we call them force-free fields. To demonstrate this process more clearly, we define   
\[
L^{2}_{\text{curl}}(\Omega)=\left\{B\in L^{2}(\Omega): B=\nabla \times A, \ A\in L^{2}(\Omega)\right\},
\]
where $\Omega=\mathbb{T}^3$ or $\Omega$ is an open and bounded domain $\Omega_0$ with zero boundary condition. Then, the minimizer of $\mathcal{E}_{B}$ over
\eqn \label{force-free admissible}
\mathcal{A}^h(\Omega)=\left\{B\in L^{2}_{\text{curl}}(\Omega): \mathcal{H}_{B}=h\right\},\quad h\in\mathbb{R}\setminus\{0\}
\een
satisfies \eqref{eq:MHD-Beltrami-state}. The existence of a magnetic energy minimizer in $\mathcal{A}^{h}(\Omega)$ is proved in \cite{Laurence Avellaneda}.

In the proof of \cite{Laurence Avellaneda, Woltjer}, the conservation of $\mathcal{H}_{B}$ is required. In \cite{Kang Lee}, the conservation of $\mathcal{H}_{B}$ and the energy
\[
\mathcal{E}=\int_{\Omega} \left(|u(t,x)|^2+|B(t,x)|^2\right)\,dx
\]
on $\Omega=\mathbb{R}^{3}$  is proved as follows: (1) $\mathcal{H}_{B}$ is conserved when $(u,B)\in L^3([0,T];L^3)$; (2) $\mathcal{E}$ is conserved when 
\[
(u,B)\in L^3\left([0,T];B^{\alpha_1}_{3,c(\mathbb{N})}\right) \times L^3\left([0,T];B^{\alpha_2}_{3,c(\mathbb{N})}\right), \quad \alpha_1\geq 1/3, \ \ \alpha_1+2\alpha_2\geq 1.
\]
(The space $B^{\alpha}_{3,c(\mathbb{N})}$ will be defined in Section \ref{sec:3} (Definition \ref{Definition 3.1}).) On the other hand, a weak solution of ideal MHD which does not conserve $\mathcal{H}_{B}$ is constructed in \cite{Beekie Buckmaster Vicol}.

\begin{remark} \upshape
In \cite{Taylor74}, Taylor showed that force-free fields with non-constant $\lambda$ come from the process minimizing the magnetic energy under local magnetic helicity conservation. In the same paper, it was conjectured that in low resistivity region ($0<\eta\ll1$) (resistive) MHD conserves the (global) magnetic helicity $\mathcal{H}_B$ in time. Recently, Taylor's conjecture was proved in \cite{FL20, FLMV22}, where a weak limit as $\nu,\eta\rightarrow0$ of Leray-Hopf weak solutions of \eqref{MHD} conserve the magnetic helicity $\mathcal{H}_B$. Following \cite{FL20} and Woltjer's variational principle, the stability of (\ref{eq:MHD-Beltrami-state}) and Chandrasekhar's force-free fields \cite{Chandrasekhar, Chandrasekhar Kendall} is established in \cite{Abe}.
\end{remark}

The force-free fields \eqref{eq:MHD-Beltrami-state} can be used to generate time-dependent force-free fields to \eqref{MHD}: 
\eqn \label{Exact MHD}
(0, \overline{B})(t,x)=(0,e^{-\eta\lambda^{2}t}\overline{B}_{0}(x)), \quad \nabla\times \overline{B}_{0}(x)=\lambda \overline{B}_{0}(x).
\een
As a follow-up to \eqref{Trkalian flow} and (\ref{Exact MHD}), we can show the stability of $\overline{u}$ in \eqref{NS} and $(0,\overline{B})$ in \eqref{MHD}. However, we do not present it since we will give similar but more complicated results to double Beltrami states in Hall MHD.

\subsection{Brief description of our results}
In Section \ref{sec:2.1}, we introduce the double Beltrami states of the form
\[
B(x)+\omega(x)=\alpha(x)u(x),\quad u(x)-J(x)=-\beta(x) B(x)
\]
as a special class of steady solutions to  the ideal Hall MHD. We then proceed to classify the double Beltrami states, and to apply Woltjer-type variational principle to derive them. In doing so, we need to introduce one more helicity, the magneto-vorticity helicity \cite{Turner}:
\[
\mathcal{H}_{B+\omega}(t)=\int_{\Omega} \left[(A+u)\cdot(B+\omega)\right](t,x)\,dx.
\]
We also answer when the energy, the magneto helicity, and the magneto-vorticity helicity are preserved. 

In Section \ref{sec:2.2}, we deal with the time-dependent double Beltrami states. We first find the exact form of the time-dependent double Beltrami states as solutions of \eqref{H MHD}. We then establish the stability of the time-dependent double Beltrami states.

Before presenting our results, we explain the domains used in this paper. When applying the variational method (Proposition \ref{Prop:2.1} and Proposition \ref{Prop:2.2}), we take bounded domains $\Omega=\mathbb{T}^3$ or $\Omega=\Omega_0$ (defined in Section \ref{sec:1.2}). When we deal with the results of conserved quantities of the ideal Hall MHD (Proposition \ref{Prop:2.3}) and stability results of the viscous and resistive Hall MHD (Theorem \ref{thm:stability1} and Theorem \ref{Theorem 2.3}), we choose $\Omega=\mathbb{T}^3$ or $\Omega=\mathbb{R}^3$. To avoid any confusion, we will specify which domain is used in each case.

\section{Double Beltrami states in Hall MHD}

\subsection{Double Beltrami states} \label{sec:2.1}
We recall the ideal Hall MHD:
\begin{subequations}\label{Ideal Hall MHD}
\begin{align}
&u_{t}  +\omega \times u +\nabla \overline{p}=J\times B,  \label{Ideal Hall MHD a}\\
&B_{t}  +  \nabla \times (B\times u)+\nabla \times \left(J\times B\right)=0,  \label{Ideal Hall MHD b}\\
&\dv u= \dv B=0, \label{Ideal Hall MHD c}
\end{align}
\end{subequations}
where  $\omega=\nabla \times u$ and $J=\nabla \times B$. We first write (\ref{Ideal Hall MHD}) as follows:
\begin{equation}\label{H MHD 1}
\begin{split}
u_{t}&=-(u-J)\times B-(B+\omega)\times u-\nabla \overline{p},\\
B_{t}&=\nabla\times\left((u-J)\times B\right).
\end{split}
\end{equation}
Then, steady solutions of (\ref{Ideal Hall MHD}) are given by divergence-free vector fields $(u,B)$ satisfying
\[
(B+\omega)\times u+\nabla p_1=0,\quad (u-J)\times B+\nabla p_2=0,\quad \overline{p}=p_1+p_2,
\]
where the scalar functions $p_1$ and $p_2$ satisfy 
\[
u\cdot \nabla p_1=(B+\omega)\cdot \nabla p_1=0,\quad B\cdot\nabla p_2=(u-J)\cdot\nabla p_2=0.
\] 

If $p_1$ and $p_2$ are constants, then the steady solutions become
\[
(B+\omega)\times u=0, \quad (u-J)\times B=0
\]
or equivalently
\eqn \label{eq:HMHD-Beltrami}
B(x)+\omega(x)=\alpha(x) u(x), \quad u(x)-J(x)=-\beta(x) B(x),
\een 
where $\alpha$ and $\beta$ are non-zero scalar functions satisfying  
\eqn \label{geometric condition}
u\cdot\nabla\alpha=(B+\omega)\cdot\nabla\alpha=0,\quad B\cdot\nabla\beta=(u-J)\cdot\nabla\beta=0
\een
due to the divergence-free condition \eqref{Ideal Hall MHD c}. We refer to divergence-free vector fields $(u,B)$ of \eqref{eq:HMHD-Beltrami} as double Beltrami states and suppose $\alpha$ and $\beta$ as at least $C^1$ functions. We can rewrite (\ref{eq:HMHD-Beltrami}) as double curl equations 
\begin{subequations}\label{Double-Beltrami}
\begin{align}
&\nabla\times\nabla\times u-(\alpha+\beta)\nabla\times u+(1+\alpha\beta)u=\nabla \alpha \times u, \label{Double-Beltrami a}\\
&\nabla\times\nabla\times B-(\alpha+\beta)\nabla\times B+(1+\alpha\beta)B=\nabla \beta \times B. \label{Double-Beltrami b}
\end{align}
\end{subequations}
Since $\nabla\times\nabla\times v=-\Delta v$ for any divergence-free vector fields $v$, the elliptic equations \eqref{Double-Beltrami} imply that double Beltrami states $(u,B)$ of \eqref{eq:HMHD-Beltrami} are $C^{k+1,\delta}$ functions if $\alpha,\beta$ are $C^{k,\delta}$ functions for $k\in\mathbb{N}$ and $\delta\in(0,1)$. Double Beltrami states of the form (\ref{eq:HMHD-Beltrami}) are introduced in \cite{Mahajan Yoshida} with constant $\alpha, \beta$, and they are special type of steady solutions of Hall MHD which do not appear in the incompressible Euler equations or ideal MHD.

\subsubsection{\bf Classification of double Beltrami states}

As a first step of investigating double Beltrami states, we classify them as superposition of two Beltrami flows. To this end, we first consider single Beltrami flows and its necessary condition satisfying \eqref{eq:HMHD-Beltrami}. Let $u$ be a Beltrami flow defined as \eqref{Beltrami flow} with $\lambda \ne 0$. Then, 
\[
\nabla \times\nabla\times u=\nabla\times(\lambda u)=\lambda\nabla\times u+\nabla \lambda\times u=\lambda^2 u+\nabla \lambda\times u.
\]
Using this and \eqref{Beltrami flow}, (\ref{Double-Beltrami a}) becomes
\eqn \label{eq:2.7}
\left[\lambda^{2}-(\alpha+\beta)\lambda +(1+\alpha\beta)\right]u=\nabla (\alpha-\lambda) \times u.
\een
By taking the inner product of this with $u$, we obtain
\[
\left[\lambda^{2}-(\alpha+\beta)\lambda +(1+\alpha\beta)\right]|u|^{2}=\left[\nabla (\alpha-\lambda) \times u\right]\cdot u=0
\]
from which  $\lambda$ satisfies 
\eqn \label{ch eq of lambda}
\lambda^{2}-(\alpha+\beta)\lambda +(1+\alpha\beta)=0.
\een
Thus, we can find two roots
\eqn \label{pm lambda}
\lambda_1,\lambda_2=\frac{\alpha+\beta\pm\sqrt{(\alpha-\beta)^{2}-4}}{2}.
\een
In addition to \eqref{ch eq of lambda}, the Beltrami flow $u$ with $\lambda$ should satisfy
\eqn \label{parallel condition}
\nabla(\alpha-\lambda)\times u=0
\een
by \eqref{eq:2.7}.

\vspace{1ex}

\noindent
$\blacktriangleright$ We first claim that $\nabla(\alpha-\beta)=0$ in the support of $u$, so the difference $\alpha-\beta$ must be constant. From \eqref{parallel condition}, $B=\alpha u-\omega=(\alpha-\lambda)u$ holds $\nabla\times B=\lambda B$ as well. Then, we substitute $B=(\alpha-\lambda)u$ into \eqref{Double-Beltrami b} to obtain 
\eqn \label{parallel condition 2}
\nabla(\beta-\lambda)\times B=(\alpha-\lambda)\nabla(\beta-\lambda)\times u=0.
\een
Since for $\lambda=\lambda_1$ or $\lambda=\lambda_2$ defined in \eqref{pm lambda}
\[
\alpha-\lambda=\frac{\alpha-\beta\mp\sqrt{(\alpha-\beta)^2-4}}{2}\neq 0,
\]
we deduce from \eqref{parallel condition} and \eqref{parallel condition 2} that
\eqn \label{parallel alpha-beta}
\nabla(\alpha-\beta)\times u=0.
\een
On the other hand, from \eqref{geometric condition} and $\nabla\times B=\lambda B$ we have
\eqn\label{orthogonal alpha-beta}
u\cdot\nabla\alpha=0,\quad u\cdot\nabla\beta=J\cdot\nabla\beta=\lambda B\cdot\nabla\beta=0.
\een
Hence, the claim $\nabla(\alpha-\beta)=0$ in the support of $u$ is followed by \eqref{parallel alpha-beta} and \eqref{orthogonal alpha-beta}.

\vspace{1ex}
\noindent
$\blacktriangleright$ Next, we find the condition of $\alpha$ and $\beta$. By multiplying \eqref{Double-Beltrami a} by $u$ with the vector calculus identity $(\nabla\times A)\cdot B=A\cdot(\nabla\times B)+\nabla\cdot(A\times B)$,
\[
|\nabla\times u|^2+\nabla\cdot((\nabla\times u) \times u)+(1+\alpha\beta)|u|^2=(\alpha+\beta)\nabla\times u\cdot u\leq |\nabla\times u|^2+\frac{(\alpha+\beta)^2}{4}|u|^2.
\]
By integrating both sides over $\Omega$, we arrive at
\[
\int_{\Omega}[(\alpha+\beta)^2-4(1+\alpha\beta)]|u|^2\,dx=\int_{\Omega}[(\alpha-\beta)^2-4]|u|^2\,dx \geq0.
\]
Hence, if a Beltrami flow $u$ satisfies \eqref{eq:HMHD-Beltrami} then $\alpha-\beta$ is constant, and so
\begin{equation} \label{Inequality of ab}
(\alpha+\beta)^2-4(1+\alpha\beta)=(\alpha-\beta)^2-4\geq 0 
\end{equation}
in the support of $u$ which also implies $\lambda_1,\lambda_2$ become real-valued functions by (\ref{pm lambda}) and \eqref{Inequality of ab}.

\vspace{1ex}

Now we are ready to present our result of a classification of double Beltrami states. We define the family of Beltrami flows with a real-valued function $\lambda$ (including the case $\lambda = 0$) by
\[
\mathcal{B}_{\lambda}=\{u\in C^1 : \nabla\times u=\lambda u,\, \dv u=0\}. 
\]
We also define the Minkowski sum of two sets: $X+Y=\{a+b : a\in X,\, b\in Y\}$.

\begin{theorem}[Classification of double Beltrami states]  \label{thm:classification}\upshape
\noindent
\begin{enumerate}[]
\item (1) Let $\lambda_1,\lambda_2$ be real-valued $C^1$ functions where the difference $\lambda_1-\lambda_2$ is constant. Then, for any $u\in \mathcal{B}_{\lambda_1}+\mathcal{B}_{\lambda_2}$ (or $B\in \mathcal{B}_{\lambda_1}+\mathcal{B}_{\lambda_2}$),  $(u,B)=(u,-\nabla\times u+\alpha u)$ (or $(u,B)=(\nabla\times B-\beta B, B)$) satisfies  \eqref{eq:HMHD-Beltrami} with
\begin{equation}\label{alpha beta}
\alpha,\beta=\frac{\lambda_1+\lambda_2\pm\sqrt{(\lambda_1-\lambda_2)^2+4}}{2}.
\end{equation}

\item (2) Let $\alpha, \beta$ be real-valued $C^1$ functions such that the difference $\alpha-\beta$ is constant and $|\alpha-\beta|>2$ so that $\lambda_1,\lambda_2$ defined in \eqref{pm lambda} are two distinct real-valued functions. Then, every double Beltrami state $(u,B)$ of \eqref{eq:HMHD-Beltrami} with $\alpha,\beta$ can be represented by a superposition of two Beltrami flows with $\lambda_1,\lambda_2$. That is, $(u,B) \in \mathcal{B}_{\lambda_1}+\mathcal{B}_{\lambda_2}.$
\vspace{1ex}
\item (3) Let $\alpha,\beta$ be real-valued $C^1$ functions such that the difference $\alpha-\beta$ is constant and $|\alpha-\beta|=2$ so that $\lambda=\lambda_1=\lambda_2=(\alpha+\beta)/2$. Then, every double Beltrami state $(u,B)$ of \eqref{eq:HMHD-Beltrami} with $\alpha,\beta$ in {$\Omega=\mathbb{T}^3$ or $\Omega= \Omega_0$}, belongs to $\mathcal{B}_{\lambda}$. 
\end{enumerate}
\end{theorem}

\begin{remark} \label{remark 2.3}\upshape
\noindent
\begin{enumerate}[]
\item \textbullet  \ It remains unclear whether a double Beltrami state of \eqref{eq:HMHD-Beltrami} with non-constant $\alpha-\beta$ exists. Hence,  a double Beltrami state may exist which cannot be represented by a superposition of two Beltrami flows.

\vspace{1ex}
\item \textbullet  \ Theorem \ref{thm:classification} (1) and (2) are domain-independent results. However, in the proof of Theorem \ref{thm:classification} (3), the $L^2$ integrability of $(u,B)$ is crucial, and this is  why we choose $\Omega=\mathbb{T}^3$ or $\Omega=\Omega_0$. By doing so, $(u,B)$ is in {$L^{2}(\Omega)$} because $(u,B)$ are $C^{1,\delta}$ functions in {a bounded domain $\Omega$}. In contrast, the $L^2$ integrability of double Beltrami states in $\mathbb{R}^3$ is not clear and we expect that there is no double Beltrami states in $L^2(\mathbb{R}^3)$. In fact, we can find a double Beltrami state $(u,B)\not\in\mathcal{B}_{\lambda}$ in $\mathbb{R}^3$. For example, 
\[
u_1(x)=
\begin{pmatrix}
Ax_{3}\sin{\lambda_0 x_{3}}+Cx_{2}\cos{\lambda_0 x_{2}}\\
Bx_{1}\sin{\lambda_0 x_{1}}+Ax_{3}\cos{\lambda_0 x_{3}}\\
Cx_{2}\sin{\lambda_0 x_{2}}+Bx_{1}\cos{\lambda_0 x_{1}}
\end{pmatrix},\ \lambda_0\in\mathbb{R},\quad 
u_2(x)=\begin{pmatrix}
x_3\sin{\frac{x_3^2}{2}} \\ x_3\cos{\frac{x_3^2}{2}} \\ 0 
\end{pmatrix}.
\]
Then, $\nabla\times u_1-\lambda_0 u_1\in \mathcal{B}_{\lambda_0}\setminus\{0\}$, $\nabla\times u_2-x_3 u_2\in\mathcal{B}_{x_3}\setminus\{0\}$ and $(u_i,\pm u_i-(\nabla\times u_i-\lambda u_i))$ satisfies \eqref{eq:HMHD-Beltrami} with $(\alpha,\beta)=(\lambda\pm1,\lambda\mp1)$.

\vspace{1ex}
\item \textbullet  \ Even considering constant $\lambda_1,\lambda_2$, the set of superpositions of two Beltrami flows $\mathcal{B}_{\lambda_1}+\mathcal{B}_{\lambda_2}$ involves a more family of divergence-free vector fields than the set of Beltrami flows $\mathcal{B}_{\lambda}$. For example, we can find a vector field $u=u_1+u_2$ in $\mathcal{B}_{\lambda_1}+\mathcal{B}_{\lambda_2}$ where $u_1$ and $u_2$ are periodic with different periods (like a superposition of ABC flows described in \eqref{ABC} with different $\lambda$). More interestingly, $\mathcal{B}_{\lambda_1}+\mathcal{B}_{\lambda_2}$ involves divergence-free vector fields on each energy shell in Fourier space. Let 
\[
u(x)=\sum_{|k|^2=\lambda^2}u_ke^{ix\cdot k},\quad  x\in\mathbb{T}^3, \ k\in\mathbb{Z}^3
\]
for arbitrary $u_k\in\mathbb{C}^3$ satisfying $\overline{u_k}=u_{-k}$ and $k\cdot u_k=0$ and $\lambda\in \{\zeta\in\mathbb{R}_{+}:\zeta^2=|k|^2 \ \text{for}\ k\in\mathbb{Z}^3\}$. Then, by the Beltrami decomposition developed in \cite{Constantin Majda}, $u\in \mathcal{B}_{\lambda}+\mathcal{B}_{-\lambda}$, which means that every divergence-free vector field obtained by Fourier modes on $\{k\in\mathbb{Z}^3:|k|^2=\lambda^2\}$ can be represented by a superposition of two Beltrami flows. Or equivalently, $u$ solves the divergence-free Helmholtz equation:
\[
-\Delta u=\lambda^2 u,\quad\dv u=0.
\]
The examples include $\sin{(\lambda x_{i})}e_{j}$, $\cos{(\lambda x_{i})}e_{j}$ for $i\neq j$, and
\[
\begin{pmatrix}
\kappa_{1}\sin{(n_{1}x_{1})}\cos{(n_{2}x_{2})}\cos{(n_{3}x_{3})}\\
\kappa_{2}\cos{(n_{1}x_{1})}\sin{(n_{2}x_{2})}\cos{(n_{3}x_{3})}\\
\kappa_{3}\cos{(n_{1}x_{1})}\cos{(n_{2}x_{2})}\sin{(n_{3}x_{3})}
\end{pmatrix},
\] 
where 
\[
\lambda^2=n^{2}_{1}+n^{2}_{2}+n^{2}_{3}, \quad \kappa_{1}n_{1}+\kappa_{2}n_{2}+\kappa_{3}n_{3}=0
\]
for $\kappa_i\in\mathbb{R}$, $n_i\in\mathbb{Z}$.
\vspace{1ex}
\item \textbullet  \ A simple example of $\mathcal{B}_{\lambda_1}+\mathcal{B}_{\lambda_2}$ with non-constant $\lambda_1,\lambda_2$, but constant $\lambda_1-\lambda_2$ is
\[
u(x)=u_1(x)+u_2(x)=\begin{pmatrix}
\sin{\frac{x_3^2}{2}} \\ \cos{\frac{x_3^2}{2}} \\ 0 
\end{pmatrix}
+
\begin{pmatrix}
\sin{\frac{(x_3+1)^2}{2}} \\ \cos{\frac{(x_3+1)^2}{2}} \\ 0 
\end{pmatrix},
\]
where $\nabla\times u_1(x)= x_3 u_1(x)$ and $\nabla\times u_2(x)=(x_3+1)u_2(x)$.
\end{enumerate}
\end{remark}

\begin{proof}[\bf Proof of Theorem \ref{thm:classification}]
(1) For $u\in\mathcal{B}_{\lambda_1}+\mathcal{B}_{\lambda_2}$, let $u=u_1+u_2$ such that $u_1\in\mathcal{B}_{\lambda_1}$ and $u_2\in\mathcal{B}_{\lambda_2}$. Then, $(u_1,(\alpha-\lambda_1)u_1)$ and $(u_2,(\alpha-\lambda_2)u_2)$ satisfy \eqref{eq:HMHD-Beltrami} with $\alpha,\beta$ given by \eqref{alpha beta}. So the proof follows by the linearity of \eqref{eq:HMHD-Beltrami}.

\vspace{1ex}

\noindent
(2) Here, we only consider $u$ because $B$ can be treated by the same process. When $\alpha-\beta$ is constant, $\alpha-\lambda_1$ and $\alpha-\lambda_2$ are constants as well for $\lambda_1,\lambda_2$ in \eqref{pm lambda}. Thus, \eqref{Double-Beltrami a} reduces to 
\[
(\nabla\times -\lambda_1)(\nabla\times -\lambda_2)u=(\nabla\times -\lambda_2)(\nabla\times -\lambda_1)u=0,\quad \dv u=0.
\]
Then, $u$  satisfies 
\[
\nabla\times u-\lambda_2 u\in\mathcal{B}_{\lambda_1},\quad
\nabla\times u-\lambda_1 u\in\mathcal{B}_{\lambda_2}.
\]
Since $|\lambda_1-\lambda_2|=\sqrt{(\alpha-\beta)^{2}-4}$ is non-zero constant, we deduce 
\[
u= \frac{1}{\lambda_1-\lambda_2}(\nabla\times u-\lambda_2 u)-\frac{1}{\lambda_1-\lambda_2}(\nabla\times u-\lambda_1 u)\in\mathcal{B}_{\lambda_1}+\mathcal{B}_{\lambda_2}.
\]
\noindent
(3) As in (2), we only consider \eqref{Double-Beltrami a} corresponding to $\lambda=\lambda_1=\lambda_2$:
\begin{equation}\label{repeated eq}
(\nabla\times -\lambda)^2u=0,\quad \dv u=0.
\end{equation}
Then by taking $L^2$ inner product of \eqref{repeated eq} with $u$ in $\Omega$ we obtain
\[
\int_{\Omega}\left|\nabla\times u-\lambda u\right|^2\,dx = 0,
\]
which implies that $\nabla\times u=\lambda u$.
This completes the proof of Theorem \ref{thm:classification}.
\end{proof}

\subsubsection{\bf Variational approach to double Beltrami states}

Now we apply Woltjer-type variational principle to find double Beltrami states. To this end, we introduce one more helicity  along with the magnetic helicity $\mathcal{H}_B$. From  (\ref{Ideal Hall MHD}), we derive the equation of the magneto-vorticity $B+\omega$:
\begin{equation}\label{B+omega eq}
(B+\omega)_{t}+\nabla\times((B+\omega)\times u)=0.
\end{equation}
The magneto-vorticity $B+\omega$ is already used in \cite{Bae Kang Shin, Chae Weng, Chae Wolf, Chae Wolf 1, Chae Wolf 2, Jeong Oh} analytically and in \cite{Meyrand, Mininni} numerically to study (\ref{H MHD}). (\ref{B+omega eq}) also yields a conserved quantity called the magneto-vorticity helicity $\mathcal{H}_{B+\omega}$  \cite{Polygiannakis Moussas, Turner}:
\[
\mathcal{H}_{B+\omega}(t)=\int_{\Omega} \left[(A+u)\cdot(B+\omega)\right](t,x)\,dx=\int_{\Omega} \left[(A+u)\cdot \nabla \times (A+u)\right](t,x)\,dx.
\]

We aim to show that an energy minimizer conserving two helicities $\mathcal{H}_{B}$ and $\mathcal{H}_{B+\omega}$ has the form of \eqref{eq:HMHD-Beltrami} with constant $\alpha,\beta$ via the variational method \cite{Abe}, \cite[Chapter 8]{Evans}. To do so, let $\Omega=\mathbb{T}^{3}$ or $\Omega=\Omega_{0}$ and we follow the definition of $L^2_{\curl}(\Omega)$ in Section \ref{sec:1.2}, and define
\[
\mathcal{A}^{h_1,h_2}(\Omega)=\left\{(B,\omega)\in L^2_{\text{curl}}(\Omega): \mathcal{H}_{B}=h_1,\  \mathcal{H}_{B+\omega}=h_2\right\}, 
\]
where $h_{1}$ and $h_{2}$ are constants.   We note that it is not clear whether there exists an energy minimizer in $\mathcal{A}^{h_1,h_2}(\Omega)$. So, we first derive (\ref{eq:HMHD-Beltrami}) under the assumption that there is a minimizer in $\mathcal{A}^{h_1,h_2}(\Omega)$.

\begin{proposition}\label{Prop:2.1} \upshape 
For $h_1,h_2\in \mathbb{R}\setminus\{0\}$, if there exists a minimizer $(B,\omega)$ of 
\[
\left\{\int_{\Omega} \left(|u(x)|^2+|B(x)|^2\right)\,dx: (B,\omega)\in\mathcal{A}^{h_1,h_2}(\Omega)\right\}
\]
over $\mathcal{A}^{h_1,h_2}(\Omega)$, then the minimizer $(B,\omega)=(B,\nabla\times u)$ satisfies
\eqn \label{u B Lagrangian}
u=\lambda_2(B+\omega),\quad u-J=-\lambda_1 B
\een
for some Lagrange multipliers $\lambda_1,\lambda_2\in\mathbb{R}$.
\end{proposition}

\begin{proof}
Since $(B,\omega)\in\mathcal{A}^{h_1,h_2}(\Omega)$ with $h_1,h_2\neq0$, there exists $(B^1,\omega^1)\in L^2_{\curl}(\Omega)$ such that
\[
\int_{\Omega}A\cdot B^1\,dx=:\delta^1\neq0,\quad \int_{\Omega}(A+u)\cdot(B^1+\omega^1)\,dx=:\delta^2\neq0.
\]
Then, we set $(B^2,\omega^2)=(B^1, B^1+2\omega^1)$ to obtain
\[
\int_{\Omega} A\cdot B^2\,dx=\delta^1,\quad \int_{\Omega}(A+u)\cdot(B^2+\omega^2)\,dx=2\delta^2.
\] 
We take arbitrary $(\widetilde{B}, \widetilde{\omega})\in L^2_{\text{curl}}(\Omega)$ and {define}
\[\begin{aligned}
j_1(\tau,s_1,s_2) &=\int_{\Omega}(A+\tau\widetilde{A}+s_1A^1+s_2A^2)\cdot(B+\tau\widetilde{B}+s_1 B^1+s_2B^2)\,dx, \\
j_2(\tau,s_1,s_2) &=\int_{\Omega}\left[(A+u)+\tau(\widetilde{A}+\widetilde{u})+s_1 (A^1+u^1)+s_2(A^2+ u^2)\right]\\
&\qquad \cdot\left[(B+\omega)+\tau(\widetilde{B}+\widetilde{\omega})+s_1 (B^1+\omega^1)+s_2 (B^2+\omega^2)\right]\,dx.
\end{aligned}\]
Then, we have $j_1(0,0,0)=h_1$, $j_2(0,0,0)=h_2$ and
\[
\begin{aligned}
&\frac{\partial j_1}{\partial\tau}(0,0,0)=2\int_{\Omega} A\cdot\widetilde{B}\,dx,\quad &&\frac{\partial j_2}{\partial\tau}(0,0,0)=2\int_{\Omega} (A+u)\cdot(\widetilde{B}+\widetilde{\omega})\,dx, \\
&\frac{\partial j_1}{\partial s_1}(0,0,0)=2 \int_{\Omega} A\cdot B^1\,dx=2\delta^1,\quad &&\frac{\partial j_2}{\partial s_1}(0,0,0)=2\int_{\Omega} (A+u)\cdot (B^1+\omega^1)\,dx=2\delta^2, \\
&\frac{\partial j_1}{\partial s_2}(0,0,0)=2 \int_{\Omega} A\cdot B^2\,dx=2\delta^1,\quad &&\frac{\partial j_2}{\partial s_2}(0,0,0)=2\int_{\Omega} (A+u)\cdot (B^2+\omega^2)\,dx=4\delta^2.
\end{aligned}
\]
Moreover, we find that  
\[
\det{\begin{pmatrix}
\partial_{s_1} j_1 & \partial_{s_2} j_1 \\ \partial_{s_1} j_2 & \partial_{s_2} j_2
\end{pmatrix}(0,0,0)}
=
\det{\begin{pmatrix} 2\delta^{1} & 2\delta^{2} \\ 2\delta^{1} & 4\delta^{2} \end{pmatrix} }
= 4\delta^1\delta^2\neq0.
\]
Hence, we deduce from the implicit function theorem that there exist $C^1$ functions $s_1(\tau)$, $s_2(\tau)$ such that 
\[
j_1(\tau,s_1(\tau),s_2(\tau))=h_1, \quad j_2(\tau,s_1(\tau),s_2(\tau))=h_2
\]
 for sufficiently small $\tau$, so that 
 \[
 (B+\tau\widetilde{B}+s_1(\tau) B^1+s_2(\tau) B^2, \omega+\tau\widetilde{\omega}+s_1(\tau)\omega^1+s_2(\tau)\omega^2)\in \mathcal{A}^{h_1,h_2}(\Omega).
 \]
 We also estimate
\[
\begin{pmatrix}
\partial_\tau j_1 \\ \partial_\tau j_2
\end{pmatrix} (0,0,0)
=
-\begin{pmatrix}
\partial_{s_1} j_1 & \partial_{s_2} j_1 \\ \partial_{s_1} j_2 & \partial_{s_2} j_2
\end{pmatrix}(0,0,0)
\begin{pmatrix}
s'_1(0) \\ s'_2(0)
\end{pmatrix}=
-2\begin{pmatrix}
\delta^{1} & \delta^{2} \\ \delta^{1} & 2\delta^{2}
\end{pmatrix}
\begin{pmatrix}
s'_1(0) \\ s'_2(0)
\end{pmatrix}
\]
to get
\[
\begin{split}
s_1'(0)&=-\frac{2}{\delta^1}\int_{\Omega} A\cdot\widetilde{B}\,dx+\frac{1}{\delta^1}\int_{\Omega}(A+u)\cdot(\widetilde{B}+\widetilde{\omega})\,dx,\\
s_2'(0)&=\frac{1}{\delta^2}\int_{\Omega} A\cdot\widetilde{B}\,dx-\frac{1}{\delta^2}\int_{\Omega}(A+u)\cdot(\widetilde{B}+\widetilde{\omega})\,dx.
\end{split}
\]
Now, we define
\[
i(\tau)=\int_{\Omega} \left(|u+\tau\widetilde{u}+s_1(\tau)u^1+s_2(\tau)u^2|^2+|B+\tau\widetilde{B}+s_1(\tau)B^1+s_2(\tau)B^2|^2\right)\,dx.
\]
Then, since $(B,\omega)\in\mathcal{A}^{h_1,h_2}(\Omega)$ is an energy minimizer, we get
\[
i'(0)=2\int_{\Omega} u\cdot(\widetilde{u}+s_1'(0)u^1+s_2'(0)u^2)\,dx +2\int_{\Omega} B\cdot(\widetilde{B}+s_1'(0)B^1+s_2'(0)B^2)\,dx=0.
\]
From this, we obtain 
\[\begin{aligned}
0=\frac{i'(0)}{2} &=\int_{\Omega} u\cdot\widetilde{u}\,dx-\lambda_{2}\int_{\Omega} (B+\omega)\cdot(\widetilde{A}+\widetilde{u})\,dx +\int_{\Omega} J\cdot\widetilde{A}\,dx-\lambda_{1}\int_{\Omega} B\cdot\widetilde{A}\,dx \\
&= \int_{\Omega} \left[u-\lambda_2(B+\omega)\right]\cdot \widetilde{u}\,dx+\int_{\Omega} \left[-\lambda_2(B+\omega)+J-\lambda_1 B\right]\cdot \widetilde{A}\,dx,
\end{aligned}\]
where
\[
\begin{aligned}
\lambda_1 &=\frac{2}{\delta^1}\left(\int_{\Omega}u\cdot u^1\,dx+\int_{\Omega}B\cdot B^1\,dx\right)-\frac{1}{\delta^2}\left(\int_{\Omega}u\cdot u^2\,dx+\int_{\Omega}B\cdot B^2\,dx\right), \\
\lambda_2 &=-\frac{1}{\delta^1}\left(\int_{\Omega}u\cdot u^1\,dx+\int_{\Omega}B\cdot B^1\,dx\right)+\frac{1}{\delta^2}\left(\int_{\Omega}u\cdot u^2\,dx+\int_{\Omega}B\cdot B^2\,dx\right).
\end{aligned}
\]
This completes the proof of Proposition \ref{Prop:2.1}.
\end{proof}

By setting $\alpha =1/\lambda_{2}$ and $\beta=\lambda_{1}$, (\ref{u B Lagrangian}) becomes (\ref{eq:HMHD-Beltrami}).  As mentioned above, we do not know whether there exists an energy minimizer in $\mathcal{A}^{h_1,h_2}$ due to the lack of the compactness of $u$. Instead, we show the existence of {a minimizer} with a prescribed $\omega\in L^2_{\curl}(\Omega)$.

\begin{proposition} \label{Prop:2.2}\upshape
For each $\omega=\nabla\times u\in L^2_{\curl}(\Omega)$ and $h_1, h_2\in\mathbb{R}\setminus\{0\}$, define
\[
\mathcal{A}_{\omega}^{h_1,h_2}(\Omega)=\left\{B\in L^2_{\text{curl}}(\Omega):  \mathcal{H}_{B}=h_1, \  \mathcal{H}_{B+\omega}=h_2\right\}.
\]
Then, there exists a minimizer of $\mathcal{E}_{B}$ over $\mathcal{A}_{\omega}^{h_1,h_2}$. Moreover, the minimizer $B$ satisfies 
\[
\nabla\times B-\lambda_1 B-\lambda_2(B+\omega)=0,
\]
for some $\lambda_1,\lambda_2\in\mathbb{R}$.
\end{proposition}

\begin{proof}
Let $\{B_j\}\subset L^2_{\text{curl}}(\Omega)$ be an energy minimizing sequence in $\mathcal{A}_{\omega}^{h_1,h_2}(\Omega)$: 
\[
\int_{\Omega}|B_j(x)|^2\,dx\longrightarrow m_B:=\inf_{B\in\mathcal{A}_{\omega}^{h_1,h_2}(\Omega)}\int_{\Omega}|B(x)|^2\,dx
\] 
as $j\rightarrow\infty$. Then, the Rellich–Kondrachov compactness theorem {\cite[Page 272]{Evans}} implies that there exists a subsequence (still denoted by $\{B_j\}$) such that
\[
B_j\rightharpoonup B \text{ in $L^2(\Omega)$},\quad A_j\rightarrow A \text{ in $L^2(\Omega)$}.
\]
Using these, we now show that the limit $B\in \mathcal{A}_{\omega}^{h_1,h_2}(\Omega)$. We first note  
\[
\left|\int_{\Omega} A\cdot B\,dx-h_1\right|=\left|\int_{\Omega} A\cdot B\,dx- \int_{\Omega} A_j\cdot B_j\,dx\right|=\left|\int_{\Omega} (A-A_j)\cdot B\,dx+\int_{\Omega} A_j\cdot(B-B_j)\,dx \right|\longrightarrow 0,
\]
and 
\[
\begin{aligned}
\left|\int_{\Omega} (A+u)\cdot\nabla\times(A+u)\,dx-h_2\right| &=\left|\int_{\Omega} (A+u)\cdot\nabla\times(A+u)\,dx-\int_{\Omega} (A_j+u)\cdot\nabla\times(A_j+u)\,dx\right| \\
&=\left|\int_{\Omega} (A-A_j)\cdot\nabla\times(A+u)\,dx+\int_{\Omega} (A_j+u)\cdot\nabla\times(A-A_j)\,dx\right|\longrightarrow 0
\end{aligned}\]
as $j\rightarrow\infty$. Moreover, by the definition of $m_B$ and {the} lower semi-continuity of weak limits, we have
\[
m_B\leq \int_{\Omega} |B(x)|^2\,dx\leq \liminf_{j\rightarrow\infty}\int_{\Omega}|B_j(x)|^2\,dx=m_{B}.
\]
Hence, the limit $B$ is an energy minimizer in $\mathcal{A}^{h_1,h_2}_{\omega}(\Omega)$.
We now show that the energy minimizer $B$ in $\mathcal{A}^{h_1,h_2}_{\omega}(\Omega)$ satisfies $\nabla\times B-\lambda_1 B-\lambda_2(B+\omega)=0$. 

\vspace{1ex}

\noindent
\underline{Case 1}. When $u=cA$ for some constant $c\in\mathbb{R}$ in the weak $L^2$ sense, $h_2$ must be $(1+c)^2 h_1$. Hence, the minimizer $B$ is reduced to an energy minimizer in $\mathcal{A}^{h_1}(\Omega)$ defined {in} \eqref{force-free admissible}, which implies $\nabla\times B-\lambda B=0$ for some constant $\lambda$. (See Section \ref{sec:1.2}.)

\vspace{1ex}

\noindent
\underline{Case 2}. When $u\neq cA$ for all $c\in\mathbb{R}$ in the weak $L^2$ sense, there exist $B^1, B^2\in L^2_{\curl}(\Omega)$ such that
\[
\int_{\Omega}A\cdot B^1\,dx\int_{\Omega}u\cdot B^2\,dx-\int_{\Omega}u\cdot B^1\,dx\int_{\Omega}A\cdot B^2\,dx=:\delta\neq0
\]
{from} which {we} arrives at
\[
\delta^{i1}:=\int_{\Omega} A\cdot B^i\,dx,\quad \delta^{i2}:=\int_{\Omega}(A+u)\cdot B^i\,dx,\quad \delta^{11}\delta^{22}-\delta^{12}\delta^{21}=\delta\neq0.
\] 
The remaining process is similar to the proof of Proposition \ref{Prop:2.1}. We take arbitrary $\widetilde{B}\in L^2_{\text{curl}}(\Omega)$ and set 
\[\begin{aligned}
j_1(\tau,s_1,s_2) &=\int_{\Omega}(A+\tau\widetilde{A}+s_1A^1+s_2A^2)\cdot(B+\tau\widetilde{B}+s_1 B^1+s_2B^2)\,dx,\\
j_2(\tau,s_1,s_2) &=\int_{\Omega}\left[(A+u)+\tau\widetilde{A}+s_1 A^1+s_2 A^2\right] \cdot\left[(B+\omega)+\tau \widetilde{B}+s_1 B^1+s_2 B^2\right]\,dx.
\end{aligned}
\]
Then, we deduce that there exist $C^1$ functions $s_1(\tau)$, $s_2(\tau)$ such that 
\[
j_1(\tau,s_1(\tau),s_2(\tau))=h_1, \quad j_2(\tau,s_1(\tau),s_2(\tau))=h_2
\]
{for sufficiently small $\tau$.} Then, we define
\[
i(\tau)=\int_{\Omega} |B+\tau\widetilde{B}+s_1(\tau)B^1+s_2(\tau)B^2|^2\,dx,
\]
which implies  
\[\begin{aligned}
0=\frac{i'(0)}{2} &=\int_{\Omega} B\cdot\widetilde{B}\,dx -\lambda_{1}\int_{\Omega} B\cdot\widetilde{A}\,dx -\lambda_{2}\int_{\Omega} (B+\omega)\cdot\widetilde{A}\,dx=\int_{\Omega} \left[\nabla\times B-\lambda_1 B-\lambda_2(B+\omega)\right]\cdot \widetilde{A}\,dx,
\end{aligned}\]
where
\[
\lambda_1 =\frac{\delta^{22}}{\delta}\int_{\Omega}B\cdot B^1\,dx-\frac{\delta^{21}}{\delta}\int_{\Omega}B\cdot B^2\,dx, \quad
\lambda_2 =-\frac{\delta^{12}}{\delta}\int_{\Omega}B\cdot B^1\,dx+\frac{\delta^{11}}{\delta}\int_{\Omega}B\cdot B^2\,dx.
\]
This completes the proof of Proposition \ref{Prop:2.2}.
\end{proof}

\subsubsection{\bf Conserved quantities}
We now in position to discuss when the energy, the magneto helicity, and the magneto-vorticity helicity are preserved. 

\begin{proposition}\label{Prop:2.3}\upshape
Let $(u,B)\in C_w([0,T];L^2({\Omega}))$ be a weak solution of \eqref{Ideal Hall MHD} in the sense of distributions in {$\Omega=\mathbb{T}^3$ or $\Omega=\mathbb{R}^3$}.
\begin{enumerate}
\item If $u\in L^3([0,T];L^3)$ and $B+\omega\in L^3([0,T];L^3)\cap C_w([0,T];\dot{H}^{-\frac{1}{2}})$, then $\mathcal{H}_{B+\omega}$ is conserved.
\item If $u\in L^3([0,T];L^3)$ and $B\in L^3([0,T];B^{\frac{1}{3}}_{3,c(\mathbb{N})})\cap C_w([0,T];\dot{H}^{-\frac{1}{2}})$, then $\mathcal{H}_B$ is conserved.
\item If $u\in L^3([0,T];B^{\frac{1}{3}}_{3,c(\mathbb{N})})$ and $B\in L^3([0,T];B^{\frac{2}{3}}_{3,c(\mathbb{N})})$, then the energy $\mathcal{E}$ is conserved.
\end{enumerate}
\end{proposition} 

Proposition \ref{Prop:2.3} will be proved in Section \ref{sec:3} after preparing some technical tools. Proposition \ref{Prop:2.3} can be proved in both $\mathbb{T}^3$ and $\mathbb{R}^3$, but we only deal with Proposition \ref{Prop:2.3} in $\mathbb{R}^{3}$ because we will use the arguments in \cite{CCFS08, Dai 4, Kang Lee}. For the same reason, we will prove Theorem \ref{thm:stability1} and Theorem \ref{Theorem 2.3} {stated below} in $\mathbb{R}^{3}$.

\subsection{Time-dependent double Beltrami states}\label{sec:2.2}
We consider the viscous and resistive Hall MHD:
\eqn \label{vH MHD}
\begin{aligned}
&u_{t}  +\omega \times u+\nabla \overline{p}-\nu\Delta u=J\times B,  \\
&B_{t}  + \nabla \times (B\times u)+\nabla \times \left(J\times B\right)- \eta\Delta B=0,   \\
&\dv u=\dv B=0.
\end{aligned}
\een

We first set $\nu=\eta>0$ which is crucial to find the exact form of time-dependent double Beltrami states of (\ref{vH MHD}). (This condition is also used to show the partial regularity or regularity of weak solutions \cite{Chae Wolf 1, Chae Wolf 2} and to simulate Hall MHD numerically \cite{Meyrand, Mininni}.)

\subsubsection{\bf Exact form of time-dependent double Beltrami states}
Under the assumption $\nu=\eta>0$, we rewrite it as (\ref{H MHD 1})
\begin{equation}\label{vH MHD 1}
\begin{split}
u_{t}-\nu\Delta u&=-(u-J)\times B-(B+\omega)\times u-\nabla \overline{p},\\
B_{t}-\nu\Delta B&=\nabla\times\left((u-J)\times B\right).
\end{split}
\end{equation}
So the nonlinear terms in (\ref{vH MHD}) are expressed in terms of $(B+\omega)\times u$ and $(u-J)\times B$. Based on this observation, we can derive the following equations
\[
\begin{split}
&(B+\omega)_{t}-\nu\Delta(B+\omega)=-\nabla \times \left((B+\omega)\times u\right),\\
&(u-J)_{t}-\nu\Delta(u-J)=-(u-J)\times B-(B+\omega)\times u-\nabla \overline{p}-\nabla \times \left(\nabla \times \left((u-J)\times B\right)\right).
\end{split}
\]
Let 
\[
\Phi=B+\omega-\alpha u, \quad\Psi=u-J+\beta B,
\]
where $\alpha$ and $\beta$ are constants. Then, 
\[
(B+\omega)\times u=\Phi\times u, \quad (u-J)\times B=\Psi\times B
\]
from which we deduce that  
\[
\begin{split}
&\Phi_{t}-\nu\Delta\Phi=-\nabla \times  (\Phi\times u)+\alpha\Psi\times B+\alpha \Phi\times u+\alpha\nabla \overline{p},\\
&\Psi_{t}-\nu\Delta\Psi=-\Psi\times B-\Phi\times u-\nabla \times \left(\nabla \times (\Psi\times B)\right)+\beta\nabla \times (\Psi\times B)-\nabla \overline{p}.
\end{split}
\]
If we initially set
\eqn \label{Beltrami initial data}
\Phi_{0}={B}_0+{\omega}_0-\alpha {u}_0=0, \quad \Psi_{0}={u}_0-{J}_0+\beta {B}_0=0,
\een
we can find one exact solution of \eqref{vH MHD} satisfying $\Phi=\Psi=0$ for all $t\geq0$.

\begin{theorem}\label{Theorem 2.1}\upshape
Let $\nu=\eta>0$ and assume that $(\overline{u}_{0}, \overline{B}_{0})$ satisfies (\ref{Beltrami initial data}) for constant $\alpha,\beta$. Then, there exists an exact solution $(\overline{u},\overline{B})$ of \eqref{vH MHD} for all $t\geq0$ starting from $(\overline{u}_0,\overline{B}_0)$ which is expressed as follows. 
\begin{enumerate}[]
\item (1) If $|\alpha-\beta|>2$, then from Theorem \ref{thm:classification} (2), $(\overline{u}_0,\overline{B}_0)$ can be expressed as 
\[
(\overline{u}_0,\overline{B}_0)=(\overline{u}^1_0,\overline{B}^1_0)+(\overline{u}^2_0,\overline{B}^2_0)\in\mathcal{B}_{\lambda_1}+\mathcal{B}_{\lambda_2}
\]
for $\lambda_1,\lambda_2$ defined in \eqref{pm lambda}. Moreover, the exact solution $(\overline{u},\overline{B})$ has the form
\[
\overline{u}(t,x)=e^{-\nu\lambda_1^2 t}\overline{u}^1_0(x)+e^{-\nu\lambda_2^2 t}\overline{u}^2_0(x), \quad \overline{B}(t,x)=e^{-\nu\lambda_1^2 t}\overline{B}^1_0(x)+e^{-\nu\lambda_2^2 t}\overline{B}^2_0(x).
\]

\item (2) If $|\alpha-\beta|=2$, then $\lambda:=\lambda_1=\lambda_2=\frac{\alpha+\beta}{2}$ and the exact solution $(\overline{u},\overline{B})$ has the form
\begin{align*}
&\overline{u}(t,x)=e^{-\nu\lambda^2 t}\overline{u}_{0}(x)-2\nu\lambda te^{-\nu\lambda^2 t}\left(\nabla\times\overline{u}_{0}(x)-\lambda\overline{u}_{0}(x)\right),\\
&\overline{B}(t,x)=e^{-\nu\lambda^2 t}\overline{B}_{0}(x)-2\nu\lambda te^{-\nu\lambda^2 t}\left(\nabla\times\overline{B}_{0}(x)-\lambda\overline{B}_{0}(x)\right).
\end{align*}
\end{enumerate}
\end{theorem}

\begin{remark} \upshape \label{remark exact form}
\noindent
\begin{enumerate}[]
\item \textbullet  \  Theorem \ref{Theorem 2.1} does not answer the uniqueness for the Cauchy problem of \eqref{vH MHD}. The uniqueness and stability for the exact solutions will be discussed in Theorem \ref{thm:stability1}.

\item \textbullet  \  Since $\alpha, \beta$ are constants, $(\overline{u}_0,\overline{B}_0)$ is smooth, and so $(\overline{u},\overline{B})$ is smooth for all $t\geq0$. However, $(\overline{u},\overline{B})$ does not belong to $L^2(\mathbb{R}^3)$ for all $t\geq0$.

\vspace{1ex}
\item \textbullet  \ When $1+\alpha\beta=\lambda_1\lambda_2\neq0$, the exact solution $(\overline{u},\overline{B})$ decays exponentially in time. In comparison, when $1+\alpha\beta=\lambda_1\lambda_2=0$ then the exact solution can be divided into two parts: exponential decaying part and non-decaying part. We note that the non-decaying part consists of harmonic vector fields. More precisely, if $1+\alpha\beta=0$ and $\alpha+\beta\neq0$, then $\lambda_1=\alpha+\beta$, $\lambda_2=0$ and so we obtain
\[
(\overline{u},\overline{B})(t,x)=e^{-\nu\lambda^2_1 t}(\overline{u}^1_0,\overline{B}^1_0)(x)+(\overline{u}^2_0,\overline{B}^2_0)(x),
\]
where $(\overline{u}^1_0,\overline{B}^1_0)\in\mathcal{B}_{\lambda_1}$ and $(\overline{u}^2_0,\overline{B}^2_0)\in \mathcal{B}_0:=\{v\in C^{\infty}:\nabla\times v=0,\, \nabla\cdot v=0\}$. Here, $v\in\mathcal{B}_0$ is a harmonic vector field due to the vector calculus identity
\[
-\Delta v=\nabla\times(\nabla\times v)-\nabla(\nabla\cdot v).
\]
Lastly if $1+\alpha\beta=\alpha+\beta=0$, then 
\[
(\overline{u},\overline{B})(t,x)=(\overline{u}_0,\overline{B}_0)(x),
\]
where $(\overline{u}_0,\overline{B}_0)$ is a pair of harmonic vector functions satisfying
\[
-\Delta \overline{u}_0=-\Delta \overline{B}_0=0,\quad \dv\overline{u}_0=\dv\overline{B}_0=0.
\]
Therefore, we conclude from the Liouville theorem for harmonic functions that if we let $(\overline{u}_0,\overline{B}_0)$ be a smooth and bounded double Beltrami state, then the non-decaying part of the exact solution $(\overline{u},\overline{B})$ must be constant.
\end{enumerate}
\end{remark}

\begin{proof}[\bf Proof of Theorem \ref{Theorem 2.1}]
Suppose that $\Phi=\Psi=0$ for all $t\geq0$. Then, from (\ref{vH MHD 1}),
\begin{align*}
&\overline{B}+\overline{\omega}-\alpha \overline{u}=0,\quad \overline{u}-\overline{J}+\beta \overline{B}=0, \label{eq5 a}\\
&\overline{u}_{t}-\nu\Delta \overline{u}=0,\quad \overline{B}_{t}-\nu\Delta \overline{B}=0.
\end{align*}
We only focus on the velocity $\overline{u}$: the magnetic field $\overline{B}$ can be obtained by the same process as for $\overline{u}$, or by just using the relation $\overline{B}=-\overline{\omega}+\alpha \overline{u}$. From \eqref{Double-Beltrami a} and \eqref{pm lambda}, we have 
\eqn \label{eq6}
\begin{aligned}
\nabla\times\nabla\times\overline{u}-(\lambda_1+\lambda_2)\nabla\times\overline{u}+\lambda_1\lambda_2\overline{u}&=(\nabla\times-\lambda_1)(\nabla\times\overline{u} -\lambda_{2}\overline{u}) \\
&=(\nabla\times-\lambda_2)(\nabla\times\overline{u} -\lambda_{1}\overline{u})=0.
\end{aligned}
\een 
Using 
\[
\overline{u}_{t}-\nu\Delta \overline{u}=\overline{u}_{t}+\nu\nabla \times \nabla\times \overline{u} =0
\]
and linearity of the equation, we have
\[
(\nabla\times\overline{u} -\lambda_{i}\overline{u})_{t}+\nu\nabla \times \nabla \times (\nabla\times\overline{u}-\lambda_{i}\overline{u})=0
\] 
for $i=1,2$. Hence, \eqref{eq6} implies 
\[
\begin{aligned}
(\nabla\times\overline{u} -\lambda_{1}\overline{u})_{t}+\nu\lambda_{2}^{2}\left(\nabla\times\overline{u} -\lambda_{1}\overline{u}\right)=0,\\
(\nabla\times\overline{u} -\lambda_{2}\overline{u})_{t}+\nu\lambda_{1}^{2}\left(\nabla\times\overline{u} -\lambda_{2}\overline{u}\right)=0.
\end{aligned}
\]
By solving this, we obtain 
\[
\begin{aligned}
(\nabla\times\overline{u}-\lambda_{1}\overline{u})(t,x)=e^{-\nu \lambda_{2}^{2}t}\left(\nabla\times\overline{u}_{0}(x)-\lambda_{1}\overline{u}_{0}(x)\right), \\
(\nabla\times\overline{u}-\lambda_{2}\overline{u})(t,x)=e^{-\nu \lambda_{1}^{2}t}\left(\nabla\times\overline{u}_{0}(x)-\lambda_{2}\overline{u}_{0}(x)\right).
\end{aligned}
\]

\noindent
$\blacktriangleright$ If $|\alpha-\beta|>2$, then $\lambda_1\neq\lambda_2$ and $\overline{u}_0=\overline{u}^1_0+\overline{u}^2_0\in\mathcal{B}_{\lambda_1}+\mathcal{B}_{\lambda_2}$, so we have
\[
\begin{split}
(\lambda_{1}-\lambda_{2})\overline{u}(t,x) &=e^{-\nu \lambda_{1}^{2}t}\left(\nabla\times\overline{u}_{0}(x)-\lambda_{2}\overline{u}_{0}(x)\right)-e^{-\nu \lambda_{2}^{2}t}\left(\nabla\times\overline{u}_{0}(x)-\lambda_{1}\overline{u}_{0}(x)\right) \\
&=(\lambda_1-\lambda_2)e^{-\nu\lambda_1^2 t}\overline{u}^1_0(x)+(\lambda_1-\lambda_2)e^{-\nu\lambda_2^2 t}\overline{u}^2_0(x)
\end{split}
\]
which implies Theorem \ref{Theorem 2.1} (1).

\vspace{1ex}
\noindent 
$\blacktriangleright$ If $|\alpha-\beta|=0$, then $\lambda=\lambda_1=\lambda_2=\frac{\alpha+\beta}{2}$, so
\[
(\nabla\times\overline{u}-\lambda \overline{u})(t,x)=e^{-\nu\lambda^2 t}(\nabla\times\overline{u}_{0}-\lambda\overline{u}_{0})(x).
\]
Also from \eqref{eq6}, we have
\[
\overline{u}_{t}=\nu\Delta \overline{u}=-\nu\nabla \times\nabla\times \overline{u} =-2\nu\lambda\nabla\times\overline{u} +\nu\lambda^2\overline{u}=-2\nu\lambda(\nabla\times\overline{u}-\lambda\overline{u})-\nu\lambda^2\overline{u}
\]
{from} which we deduce that 
\[
\frac{d}{dt}\left(e^{\nu\lambda^2 t}\overline{u}\right)=-2\nu\lambda e^{\nu\lambda^2 t}(\nabla\times\overline{u}-\lambda\overline{u})=-2\nu\lambda (\nabla\times\overline{u}_{0}(x)-\lambda\overline{u}_{0}(x))
\]
which implies Theorem \ref{Theorem 2.1} (2).
\end{proof}

\subsubsection{\bf Stability of time-dependent double Beltrami states}

Since we find exact solutions in Theorem \ref{Theorem 2.1}, it is natural to study the stability of these solutions. Suppose that $\nu=\eta>0$ and $(\overline{u}_0,\overline{B}_0)$ is a smooth and bounded double Beltrami state of \eqref{Beltrami initial data} with constant factors $\alpha,\beta$. Then, the time-dependent double Beltrami state $(\overline{u},\overline{B})(t)$ defined in Theorem \ref{Theorem 2.1} with initial data $(\overline{u}_0,\overline{B}_0)$ is also a smooth and bounded solution, and the temporal non-decaying part of $(\overline{u},\overline{B})$ must be constant from Remark \ref{remark exact form}.

\begin{theorem}\upshape \label{thm:stability1}
Let $\nu=\eta>0$ and $(\overline{u}_0,\overline{B}_0)$ be a smooth and bounded double Beltrami state of \eqref{Beltrami initial data} with constant factors $\alpha,\beta$. Suppose $(u_0,B_0)$ can be expressed as $(u_0,B_0)=(\overline{u}_0,\overline{B}_0)+(v_0,b_0)$ where $(v_0,b_0)\in H^3(\mathbb{R}^3)$ with sufficiently small $\|v_0\|_{H^3}+\|b_0\|_{H^3}$. Then, there exists a global unique solution 
\[
(u,B)=(\overline{u},\overline{B})+(v,b),
\] 
where $(\overline{u},\overline{B})$ is the time-dependent double Beltrami state established in Theorem \ref{Theorem 2.1} with initial data $(\overline{u}_0,\overline{B}_0)$ and $(v,b)$ is the global unique solution of \eqref{eq:perturbed-HMHD} in $C([0,\infty);H^3)$ satisfying
\eqn \label{perturbed eq bounds}
\norm{v(t)}{H^{3}}^{2}+\norm{b(t)}{H^{3}}^{2}+\nu\int^{t}_{0} \left(\norm{\nabla v(\tau)}{H^{3}}^{2} +\norm{\nabla b(\tau)}{H^{3}}^{2}\right)\,d\tau \leq C(\|v_0\|_{H^3}^2+\|b_0\|_{H^3}^2).
\een
Moreover, we derive the asymptotic behavior of $(u,B)$ as follows:
\[
\sup_{x\in\mathbb{R}^3}\left|(u,B)-(U,M)\right|\leq C(1+t)^{-\frac{3}{4}}\quad\text{for all $t\geq0$},
\]
where $(U,M)$ is a pair of constant vectors in temporal non-decaying part of $(\overline{u},\overline{B})$ which become trivial when $1+\alpha\beta\neq0$.
\end{theorem}

Theorem \ref{thm:stability1} is a direct result from Theorem \ref{Theorem 2.1} and the  two propositions {stated below}. To show the propositions, we suppose that $(\overline{u},\overline{B})=(U,M)+(\overline{v},\overline{b})$ where $(U,M)$ is a pair of constant vectors (non-decaying part) and $(\overline{v},\overline{b})$ is a pair of smooth and bounded vector fields for all $t\geq0$ (decaying part). If we set $(v,b)=(u-\overline{u},B-\overline{B})$, then $(v,b)$ satisfies 
\eqn\label{eq:perturbed-HMHD}
\begin{split}
&v_{t}-\nu\Delta v+(v\cdot\nabla) v-(b\cdot\nabla) b+\nabla q +(U\cdot\nabla)v-(M\cdot\nabla)b=-(v\cdot\nabla)\overline{v}-(\overline{v}\cdot\nabla) v+(b\cdot\nabla)\overline{b}+(\overline{b}\cdot\nabla) b,\\
&b_{t}-\nu\Delta b+(v\cdot\nabla) b-(b\cdot\nabla) v +\nabla\times((\nabla\times b)\times b)+(U\cdot\nabla)b-(M\cdot\nabla)v+\nabla\times((\nabla\times b)\times M) \\
&=-(v\cdot\nabla)\overline{b}-(\overline{v}\cdot\nabla) b+(b\cdot\nabla)\overline{v}+(\overline{b}\cdot\nabla) v -\nabla\times((\nabla\times b)\times\overline{b})-\nabla\times((\nabla\times\overline{b})\times b),\\
&\dv  v=\dv b=0,\\
&(v_{0},b_{0})=(u_{0}-\overline{u}_{0},B_{0}-\overline{B}_{0}).
\end{split}
\een 

Since $(\overline{v},\overline{b})$ is smooth, but not necessarily in $L^2(\mathbb{R}^{3})$, we control $(\overline{v},\overline{b})$ only using $L^\infty$-type norms. In fact, $(\overline{v},\overline{b})$ and its derivatives decay exponentially in time from Theorem \ref{Theorem 2.1} and thus we can control the right-hand side of \eqref{eq:perturbed-HMHD} using the uniform bound of 
\[
\int^{T}_{0}\left(\norm{\overline{v}(\tau)}{W^{3,\infty}}^{2}+\norm{\overline{b}(\tau)}{W^{4,\infty}}^{2}\right)\,d\tau
\]
for all $T>0$. On the one hand, we bound the left-hand side of \eqref{eq:perturbed-HMHD} following \cite{Chae Degond Liu}.

\begin{proposition}\label{Prop perturbed gwp}\upshape
Let $(v_{0},b_{0})\in H^{3}(\mathbb{R}^3)$ with $\dv v_{0}=\dv b_{0}=0$. Suppose $(\overline{v}, \overline{b})$
satisfies
\[
\int^{\infty}_{0}\left(\norm{\overline{v}(\tau)}{W^{3,\infty}}^{2}+\norm{\overline{b}(\tau)}{W^{4,\infty}}^{2}\right)\,d\tau\leq C_{0},
\]
for some $C_0>0$. Then, there exits $\epsilon>0$ such that if $\|v_{0}\|_{H^{3}} +\|b_{0}\|_{H^{3}}<\epsilon$, there exists a global unique solution $(v,b)\in C([0,\infty);H^3)$ of \eqref{eq:perturbed-HMHD} satisfying \eqref{perturbed eq bounds} for all $t\geq0$.
\end{proposition}

The smallness assumption of $(v_0,b_0)$ in Proposition \ref{Prop perturbed gwp} is necessary to find $(v,b)$ globally-in-time. On the other hand, if we choose $(v_{0},b_{0})\in L^{2}$, we can derive a weak solution satisfying
\eqn \label{perturbed weak-ineq}
\left\|v(t)\right\|^{2}_{L^{2}}+\left\|b(t)\right\|^{2}_{L^{2}}+\nu\int^{t}_{0}\left(\left\|\nabla v(\tau)\right\|^{2}_{L^{2}}+\left\|\nabla b(\tau)\right\|^{2}_{L^{2}}\right)d\tau\leq \left(\left\|v_{0}\right\|^{2}_{L^{2}}+\left\|b_{0}\right\|^{2}_{L^{2}}\right)\exp\left[\frac{C_{0}}{\nu}\right]
\een
without assuming the smallness condition of $(v_{0},b_{0})$ in $L^{2}$, where $C_{0}$ comes from
\[
\int^{\infty}_{0}\left(\norm{\overline{v}(\tau)}{L^{\infty}}^{2}+\norm{\overline{b}(\tau)}{W^{1,\infty}}^{2}\right)\,d\tau\leq C_{0}.
\]

We can also derive temporal decay rates of $(v,b)$ if we impose temporal decay rates of $(\overline{v},\overline{b})$.

\begin{proposition}\upshape \label{Prop perturbed decay}
Under the assumptions of Proposition \ref{Prop perturbed gwp}, if $(\overline{v},\overline{b})$ additionally satisfies
\eqn \label{perturbed decay condition}
\|\overline{v}(t)\|^2_{W^{3,\infty}}+\|\overline{b}(t)\|^2_{W^{4,\infty}}\leq C(1+t)^{-4}
\een
for all $t\geq0$ for some $C>0$, then $(v,b)$ decays in time as follows. For $k\in(0,3]$, there exists $C_k,\overline{C}>0$ such that
\begin{subequations}
\begin{align}
&\left\|v(t)\right\|^{2}_{\dot{H}^{k}}+\left\|b(t)\right\|^{2}_{\dot{H}^{k}} \leq C_k (1+t)^{-k}, \label{perturbed decay1} \\
&\left\|v(t)\right\|_{L^{\infty}}+\left\|b(t)\right\|_{L^{\infty}} \leq \overline{C}(1+t)^{-\frac{3}{4}} \label{perturbed decay2}
\end{align}
\end{subequations}
for all $t\ge 0$.
\end{proposition}

Since the proof of Theorem \ref{thm:stability1} is based on the existence of time-dependent double Beltrami states in Theorem \ref{Theorem 2.1}, it is necessary to suppose $\nu=\eta>0$. To extend the stability result into more physically reasonable assumption $\nu\sim\eta$ \cite[Remark 2.3]{He Huang Wang}, we use the quantities $\Phi=B+\omega-\alpha u$ and $\Psi=u-J+\beta B$ measuring how close solutions are from the structure of \eqref{eq:HMHD-Beltrami}. More precisely, we take the smallness condition of $(\Phi_{0}, \Psi_{0})$ to show the global-in-time well-posedness of solutions to \eqref{vH MHD}, so $(u_{0}, B_{0})$ can be arbitrarily large. The following statement holds for both $\mathbb{T}^3$ and $\mathbb{R}^3$, but we state it only for $\mathbb{R}^3$.

\begin{theorem}\label{Theorem 2.3} \upshape
Let $\nu,\eta>0$ satisfy $16|\nu-\eta|\leq \nu+\eta$ and $\mu=\frac{\nu+\eta}{2}$. Let $(u_{0},B_{0})\in H^{3}(\mathbb{R}^3)$ with $\dv u_{0}=\dv B_{0}=0$ and $\alpha, \beta \in \mathbb{R}$  with $1+\alpha\beta\neq0$. If $(u_{0}, B_{0})$ satisfies
\[
\left(\norm{\Phi_{0}}{H^{\frac{1}{2}}}^{2}+\norm{\Psi_{0}}{H^{\frac{1}{2}}}^{2}+C_3\mathcal{E}_0\frac{(\nu-\eta)^2}{(\nu+\eta)^2}\right)\exp{\left(\frac{C_1\mathcal{E}_0}{\mu^2}+\frac{C_1\mathcal{E}_0^2}{\mu^4}\right)}<\frac{\mu^2}{C_2^2},
\]
then there exists a global unique solution $(u,B)\in C([0,\infty); H^{3})$ of (\ref{vH MHD}) satisfying
\[
\norm{\Phi(t)}{H^{\frac{1}{2}}}^{2}+\norm{\Psi(t)}{H^{\frac{1}{2}}}^{2}+\frac{\mu}{2}\int^{t}_{0} \left(\norm{\nabla \Phi(\tau)}{H^{\frac{1}{2}}}^{2}+\norm{\nabla \Psi(\tau)}{H^{\frac{1}{2}}}^{2}\right)d\tau <\frac{\mu^2}{C_2^2}
\]
for all $t>0$. Here, $\mathcal{E}_0=\norm{u_0}{L^2}^2+\norm{B_0}{L^2}^2$ and $C_1, C_2, C_3$, fixed in (\ref{eq:4.7}), depend on $\alpha,\beta$. 
\end{theorem}

\begin{remark}\upshape
We use the condition $1+\alpha\beta\neq0$ only  in (\ref{I 1}) during the proof of Theorem \ref{Theorem 2.3}. If we take $\Omega=\mathbb{T}^{3}$ with mean-zero condition of $(u_0,B_0)$, we can rule out the condition $1+\alpha\beta\neq0$. Since the mean-zero condition to $(u_{0}, B_{0})$ implies the mean-zero condition to $(u, B)$, we can use the Poincar\'e inequality to bound $\text{I}_{1}$ in (\ref{L2 bound phi psi}) as follows: 
\[
\text{I}_{1}\leq C\left(\norm{\nabla u}{L^{2}}+\norm{\nabla B}{L^{2}}\right)\norm{\Psi}{L^{3}}\norm{\Phi}{L^{6}}.
\]
Moreover, once we get the global existence in $H^3(\mathbb{T}^3)$, it is easy to check the exponential decay of $\left\|(u,B)(t)\right\|_{H^3}$ by using Poincar\'e inequality.
\end{remark}

Before closing this section, we fix some notations. All  constants will be denoted by $C$ and we follow the convention that such constants can vary from expression to expression and even between two occurrences within the same expression.  In Section \ref{sec:3}-\ref{sec:5}, we  use the simplified form of the integral of the spatial variables:
\[
\int=:\int_{\mathbb{R}^3}dx.
\]
For two vector fields $V$ and $W$ in $\mathbb{R}^{3}$, we define a matrix $V\otimes W$ whose $(i,j)$-the entry {is} $V_{i}W_{j}$.

\section{Proof of Proposition \ref{Prop:2.3} }\label{sec:3}
In this section, we prove Proposition \ref{Prop:2.3} using the Littlewood-Paley decomposition and Besov spaces in $\mathbb{R}^{3}$. We can also define the Littlewood-Paley theory on $\mathbb{T}^{3}$ \cite{Dai Hu Wu Xiao}, but we choose $\mathbb{R}^{3}$ because this theory is more familiar in $\mathbb{R}^{3}$ than $\mathbb{T}^{3}$.

\subsection{Littlewood-Paley decomposition}
Let $\mathcal{C}=\left\{\xi\in \mathbb{R}^{d}: \frac{3}{4}\leq |\xi|\leq \frac{8}{3}\right\}$ and take smooth radial functions $\chi$ and $\Phi$ with values in $[0,1]$ supported on $B(0,\frac{4}{3})$ and $\mathcal{C}$ respectively, and satisfy
\[
\chi(\xi)+\sum^{\infty}_{q=0}\Phi\left(\lambda_q^{-1}\xi\right)=1 \quad  \forall \ \xi \in \mathbb{R}^{d},
\]
where $\lambda_q=2^q$ and  $\Phi(\xi)=\chi(\xi/2)-\chi(\xi)$.
Using this, we define dyadic blocks:
\begin{align*}
&\Delta_{q}f=\lambda_q^{d} \int_{\mathbb{R}^{d}} h\left(\lambda_q y\right)f(x-y)\,dy, \quad h=\mathcal{F}^{-1}\Phi\quad q\geq0,\\
&\Delta_{-1}f=\int_{\mathbb{R}^{d}} \widetilde{h}(y))f(x-y)\,dy, \quad \widetilde{h}=\mathcal{F}^{-1}\chi.
\end{align*}
Then, the inhomogeneous Littlewood-Paley decomposition is given by
\[
f=\sum_{q=-1}^{\infty} \Delta_{q}f \ \  \text{in} \ \  \mathcal{S}^{'},
\]
where $\mathcal{S}^{'}$ is the space of tempered distributions. For $Q\geq -1$, we define
\[
S_Q f=\sum^{Q}_{q=-1}\Delta_q f=\int_{\mathbb{R}^{d}} \widetilde{h}_Q\left(y\right)f(x-y)\,dy, \quad \widetilde{h}_{Q}(x)=\lambda_Q^d \widetilde{h}(\lambda_Q x).
\]
We recall two properties of the Littlewood-Paley decomposition: 
when $f\in L^p$ 
\begin{itemize}
\item for $1<p<\infty$
\begin{equation}\label{dyadic approx in Lp}
\norm{S_Q f-f}{L^p}\rightarrow 0 \quad\text{as}\quad Q\rightarrow\infty,
\end{equation}
\item for $1\leq p\leq \infty$
\eqn \label{Bernstein ineq}
\|\nabla\Delta_q f\|_{L^p}\leq C\lambda_q \|\Delta_q f\|_{L^p}.
\een
\end{itemize}


\begin{definition} \label{Definition 3.1}\upshape
We define $B^\alpha_{3,c(\mathbb{N})}$ to be the class of tempered distributions for which
\[
\lim_{q\rightarrow\infty}\lambda_q^\alpha\norm{\Delta_q f}{L^3}=0,
\]
and with the norm inherited from $B^\alpha_{3,\infty}$.
\end{definition}

\subsection{Proof of Proposition \ref{Prop:2.3}}
Let
\begin{align*}
\text{the energy}:\quad & \mathcal{E}(t)=\frac{1}{2}\int \left(|u(t,x)|^2+|B(t,x)|^2\right)\\
\text{the magnetic helicity}:\quad & \mathcal{H}_B(t)=\int (A\cdot B)(t,x) \\
\text{the magneto-vorticity helicity}:\quad & \mathcal{H}_{B+\omega}(t)=\int \left[(A+u)\cdot(B+\omega)\right](t,x),
\end{align*}
where $\nabla\times A =B$ and $\dv A=0$. From \eqref{Ideal Hall MHD}, we {derive} the following equations:
\begin{subequations} \label{ideal Hall 2}
\begin{align}
& A_{t} +B\times u +(\nabla\times B)\times B+\nabla q=0, \label{ideal Hall d}\\
& (A+u)_{t}+(B+\omega)\times u+\nabla(\overline{p}+q)=0, \label{ideal Hall e}\\
& (B+\omega)_{t}+\nabla\times((B+\omega)\times u)=0. \label{ideal Hall f}
\end{align}
\end{subequations}

We follow the method of \cite{CCFS08}, which is also applied to the ideal MHD in \cite{Kang Lee}, and the electron MHD in \cite{Dai 4}. We first introduce the localization kernels
\[
K_1(q)=
\begin{cases}
\lambda_q^{\frac{2}{3}}\quad q\leq 0 \\
\lambda_q^{-\frac{4}{3}}\quad q>0,
\end{cases} \ K_2(q)=
\begin{cases}
\lambda_q^{\frac{4}{3}}\quad q\leq 0 \\
\lambda_q^{-\frac{2}{3}}\quad q>0,
\end{cases}
\]
and define
\[
b_q=\lambda_q^{\frac{1}{3}}\norm{\Delta_q B}{L^3},\quad d_q=\lambda_q^{\frac{1}{3}}\norm{\Delta_q u}{L^3},\quad e_q=\lambda_q^{\frac{2}{3}}\norm{\Delta_q B}{L^3},
\]
and 
\[
b^2=\{b_q^2\}_{q\geq-1},\quad d^2=\{d_q^2\}_{q\geq-1},\quad e^2=\{e_q^2\}_{q\geq-1}
\]
which will be used to prove (2) and (3) of Proposition \ref{Prop:2.3}. {We omit all generic constants $C>0$ in this section for simplicity.}

\vspace{1ex}

\noindent
$\blacktriangleright$ From \eqref{ideal Hall e} and \eqref{ideal Hall f}, we deal with the truncated magneto-vorticity helicity $\mathcal{H}_{B+\omega}^Q$ as
\[
\frac{d}{dt}\mathcal{H}_{B+\omega}^{Q}=\frac{d}{dt}\int S_Q(A+u)\cdot S_Q(B+\omega)=-2\int S_Q((B+\omega)\times u)\cdot S_Q(B+\omega).
\]
Then
\begin{align*}
&\int S_Q((B+\omega)\times u)\cdot S_Q(B+\omega) \\
&=\int\left[S_Q((B+\omega)\times u)-(B+\omega)\times u\right]\cdot S_Q(B+\omega) +\int ((B+\omega)\times u)\cdot \left[S_Q(B+\omega)-(B+\omega)\right] \\
&\leq \norm{S_Q((B+\omega)\times u)-(B+\omega)\times u}{L^{3/2}}\norm{S_Q(B+\omega)}{L^3}+\norm{(B+\omega)\times u}{L^{3/2}}\norm{S_Q(B+\omega)-(B+\omega)}{L^3}
\end{align*}
converges to zero as $Q\rightarrow\infty$ by using \eqref{dyadic approx in Lp} with $u, B+\omega \in L^3([0,T];L^3)$.
Hence 
\[
\frac{d}{dt}\mathcal{H}_{B+\omega}^{Q}\rightarrow 0\quad \text{as}\quad Q\rightarrow\infty,
\]
which completes the proof of Proposition \ref{Prop:2.3} (1).

\vspace{1ex}

\noindent
$\blacktriangleright$ Similarly, we consider the truncated magnetic helicity $\mathcal{H}_B^Q$ as
\begin{align*}
\frac{d}{dt}\mathcal{H}_B^Q=\frac{d}{dt}\int S_Q A\cdot S_Q B &=2\int S_Q(u\times B)\cdot S_Q B-2\int S_Q((\nabla\times B)\times B)\cdot S_Q B=2\text{I}(Q)+2\text{II}(Q).
\end{align*}
As in the proof of Proposition \ref{Prop:2.3} (1), $\text{I}(Q)$ vanishes to zero as $Q\rightarrow\infty$: for $u, B \in L^3([0,T];L^3)$
\begin{align*}
\text{I}(Q)&=\int\left[S_Q(u\times B)-u\times B\right]\cdot S_Q B+\int(u\times B)\cdot(S_Q B-B) \\
&\leq \norm{S_Q(u\times B)-u\times B}{L^{3/2}}\norm{S_Q B}{L^3}+\norm{u\times B}{L^{3/2}}\norm{S_Q B-B}{L^3}\rightarrow 0.
\end{align*}
Next we compute $\text{II}(Q)$ as
\begin{align*}
\text{II}(Q)&=-\int S_Q((\nabla\times B)\times B)\cdot S_Q B \\
&=-\int S_Q(\nabla\cdot(B\otimes B)-\frac{1}{2}\nabla |B|^2)\cdot S_Q B =\int S_Q(B\otimes B):\nabla S_Q B.
\end{align*}
We now divide $S_Q (B\otimes B)$ into 3 terms following \cite{CCFS08}:
\begin{equation}\label{division of quadratic}
\begin{split}
& S_Q(f\otimes g)=r_Q(f,g)-(f-S_Q f)\otimes(g-S_Q g)+S_Q f\otimes S_Q g, \\
& r_Q(f,g)=\int_{\mathbb{R}^3} \widetilde{h}_{Q}(y)(f(x-y)-f(x))\otimes(g(x-y)-g(x))\,dy.
\end{split}
\end{equation}
From this, 
\[
\text{II}(Q)=\int r_Q(B,B):\nabla S_Q B-\int(B-S_Q B)\otimes(B-S_Q B):\nabla S_Q B=\text{II}_1(Q)+\text{II}_2(Q),
\]
where we remove
\[
\int \left(S_Q B\otimes S_Q B\right):\nabla S_Q B=0.
\]
We first bound $\text{II}_1(Q)$ as
\[
\text{II}_1(Q)\leq \norm{r_Q(B,B)}{L^{3/2}}\norm{\nabla S_Q B}{L^3}
\]
and 
\[
\norm{r_Q(B,B)}{L^{3/2}}\leq \int_{\mathbb{R}^3} \left|\widetilde{h}_Q(y)\right|\norm{B(\cdot-y)-B(\cdot)}{L^3}^2\,dy.
\]
{We use \eqref{Bernstein ineq} to get}
\begin{align*}
\norm{B(\cdot-y)-B(\cdot)}{L^3}^2 &\leq \sum_{q\leq Q}|y|^2\lambda_q^2\norm{\Delta_q B}{L^3}^2+\sum_{q>Q}\norm{\Delta_q B}{L^3}^2 \\
&\leq\lambda_Q^{\frac{4}{3}}|y|^2\sum_{q\leq Q}\lambda_{Q-q}^{-\frac{4}{3}}\left(\lambda_q^{\frac{1}{3}}\norm{\Delta_q B}{L^3}\right)^2+\lambda_Q^{-\frac{2}{3}}\sum_{q>Q}\lambda_{Q-q}^{\frac{2}{3}}\left(\lambda_q^{\frac{1}{3}}\norm{\Delta_q B}{L^3}\right)^2 \\
&\leq \left(\lambda_Q^{\frac{4}{3}}|y|^2+\lambda_Q^{-\frac{2}{3}}\right)\left(K_1\ast {b^2}\right)(Q).
\end{align*}
{Hence, we arrive at}
\begin{align*}
\text{II}_1(Q)&\leq \left(K_1\ast {b^2}\right)(Q)\int_{\mathbb{R}^3}\left|\widetilde{h}_Q(y)\right|\left(\lambda_Q^{\frac{4}{3}}|y|^2+\lambda_Q^{-\frac{2}{3}}\right)\,dy\, \norm{\nabla S_Q B}{L^3} \\
&\leq \left(K_1\ast {b^2}\right)(Q) \lambda_Q^{-\frac{2}{3}}\Bigg(\sum_{q\leq Q}\lambda_q^2\norm{\Delta_q B}{L^3}^2\Bigg)^{\frac{1}{2}} \leq \left(K_1\ast {b^2}\right)^{\frac{3}{2}}(Q).
\end{align*}
{Similarly, $\text{II}_2(Q)$ is estimated as}
\begin{align*}
\text{II}_2(Q)&\leq \norm{B-S_Q B}{L^3}^2\norm{\nabla S_Q B}{L^3} \leq \Bigg(\sum_{q>Q}\norm{\Delta_q B}{L^3}^2 \Bigg) \Bigg(\sum_{q\leq Q}\lambda_q^2\norm{\Delta_q B}{L^3}^2 \Bigg)^{\frac{1}{2}} \\
&=\Bigg(\sum_{q>Q}\lambda_{Q-q}^{\frac{2}{3}}b_q^2 \Bigg) \Bigg(\sum_{q\leq Q}\lambda_{Q-q}^{-\frac{4}{3}}b_q^2 \Bigg)^{\frac{1}{2}}\leq \left(K_1\ast {b^2}\right)^{\frac{3}{2}}(Q).
\end{align*}
When $B\in L^3([0,T];B^{\frac{1}{3}}_{3,c(\mathbb{N})})$, $\text{II}(Q)$ vanishes as $Q\rightarrow\infty$. Therefore, $\mathcal{H}_{B}$ is conserved in time.

\vspace{1ex}

\noindent
$\blacktriangleright$ We finally prove the conservation of the energy $\mathcal{E}$.
\begin{align*}
\frac{d}{dt}\mathcal{E}^Q =&\frac{1}{2}\frac{d}{dt}\int |S_Q u|^2+|S_Q B|^2 \\
=& \left(-\int S_Q\nabla\cdot(u\otimes u-B\otimes B)\cdot S_Q u +\int S_Q\nabla\cdot(B\otimes u-u\otimes B)\cdot S_Q B\right)\\
&-\int S_Q((\nabla\times B)\times B)\cdot S_Q(\nabla\times B) =\text{III}(Q)+\text{IV}(Q).
\end{align*}
By dividing $S_Q(f\otimes g)$ as in \eqref{division of quadratic}, we write
\begin{align*}
\text{III}(Q)=& \int \left[r_Q(u,u)-r_Q(B,B)\right]:\nabla S_Q u -\int\left[r_Q(B,u)-r_Q(u,B)\right]:\nabla S_Q B  \\
&-\int\left[(u-S_Q u)\otimes(u-S_Q u)-(B-S_Q B)\otimes(B-S_Q B)\right]:\nabla S_Q u \\
&+\int\left[(B-S_Q B)\otimes(u-S_Q u)-(u-S_Q u)\otimes(B-S_Q B)\right]:\nabla S_Q B .
\end{align*}
Following the same process to estimate $\text{II}(Q)$, we conclude
\[
\text{III}(Q)\leq \left(K_1\ast {d^2}\right)^{\frac{3}{2}}(Q)+\left(K_1\ast{d^2}\right)^{\frac{1}{2}}(Q)\left(K_1\ast {b^2}\right)(Q)\rightarrow 0
\]
as $Q\rightarrow\infty$ for $(u,B)\in L^3([0,T];B^{\frac{1}{3}}_{3,c(\mathbb{N})})$.
Again using \eqref{division of quadratic}, we estimate
\begin{align*}
\text{IV}(Q) &=\int S_Q(B\otimes B):\nabla\nabla\times S_Q B \\
&=\int r_Q(B,B):\nabla\nabla\times S_Q B-\int(B-S_Q B)\otimes(B-S_Q B):\nabla\nabla\times S_Q B =\text{IV}_1(Q)+\text{IV}_2(Q),
\end{align*}
where we use
\[
\int S_Q B\otimes S_Q B:\nabla\nabla\times S_Q B=-\int\left[(\nabla\times S_Q B)\times S_Q B+\frac{1}{2}\nabla|S_Q B|^2\right]\cdot\nabla\times S_Q B=0.
\]
Since
\[
\norm{B(\cdot-y)-B(\cdot)}{L^3}^2 \leq \sum_{q\leq Q}|y|^2\lambda_q^2\norm{\Delta_q B}{L^3}^2+\sum_{q>Q}\norm{\Delta_q B}{L^3}^2 \leq \left(\lambda_Q^{\frac{2}{3}}|y|^2+\lambda_Q^{-\frac{4}{3}}\right)\left(K_2\ast {e^2}\right)(Q),
\]
we bound $\text{IV}_1(Q)$ as
\begin{align*}
\text{IV}_1(Q) &\leq \left(K_2\ast e^2\right)(Q)\int_{\mathbb{R}^3}\left|\widetilde{h}_Q(y)\right|\left(\lambda_Q^{\frac{2}{3}}|y|^2+\lambda_Q^{-\frac{4}{3}}\right)\,dy\,\norm{\nabla\nabla\times S_Q B}{L^3} \\
&\leq \left(K_2\ast e^2\right)(Q)\lambda_Q^{-\frac{4}{3}} \Bigg(\sum_{q\leq Q}\lambda_q^4\norm{\Delta_q B}{L^3}^2 \Bigg)^{\frac{1}{2}} \\
&\leq \left(K_2\ast e^2\right)(Q) \Bigg(\sum_{q\leq Q}\lambda_{Q-q}^{-\frac{8}{3}}\left(\lambda_q^{\frac{2}{3}}\norm{\Delta_q B}{L^3}\right)^2 \Bigg)^{\frac{1}{2}}\leq \left(K_2\ast {e^2}\right)^{\frac{3}{2}}(Q).
\end{align*}
{Similarly, $\text{IV}_2(Q)$ is estimated as}
\begin{align*}
\text{IV}_2(Q) &\leq \norm{B-S_Q B}{L^3}^2\norm{\nabla\nabla S_Q B}{L^3} \leq \left(\sum_{q>Q}\norm{\Delta_q B}{L^3}^2\right) \Bigg(\sum_{q\leq Q}\lambda_q^4\norm{\Delta_q B}{L^3}^2 \Bigg)^{\frac{1}{2}} \\
&\leq \Bigg(\sum_{q>Q}\lambda_{Q-q}^{\frac{4}{3}} e_q^2 \Bigg) \Bigg(\sum_{q\leq Q}\lambda_{Q-q}^{-\frac{8}{3}} e_q^2 \Bigg)^{\frac{1}{2}}\leq \left(K_2\ast {e^2}\right)^{\frac{3}{2}}(Q).
\end{align*}
Thus, when $B\in L^3([0,T];B^{\frac{2}{3}}_{3,c(\mathbb{N})}(\mathbb{R}^3))$, $\text{IV(Q)}$ converges to zero as $Q\rightarrow0$ which completes the proof of Proposition \ref{Prop:2.3}.

\section{Proof of Proposition \ref{Prop perturbed gwp} and Proposition \ref{Prop perturbed decay}} \label{sec:4}

\subsection{Function spaces and some inequalities}
Since we prove Proposition \ref{Prop perturbed gwp} and Proposition \ref{Prop perturbed decay} in this section and Theorem \ref{Theorem 2.3} in $\mathbb{R}^{3}$ in Section \ref{sec:5}, we define function spaces and some inequalities in $\mathbb{R}^{3}$.

We first introduce notation of derivatives (in any dimension greater than or equal to 2). Let $\alpha=(\alpha_{1}, \alpha_{2}, \cdots, \alpha_{d})$, $\alpha_{i}\ge0$, $i=1,2,\cdots, d$, be a multi-index of order $l\in \mathbb{N}$: $l=\alpha_{1}+\alpha_{2}+\cdots +\alpha_{d}=:|\alpha|$. Then, we define
\[
\nabla^{l}f(x)=\sum_{|\alpha|=l}\frac{\partial^{l}f(x)}{\partial^{\alpha_{1}}_{x_{1}} \partial^{\alpha_{2}}_{x_{2}} \cdots \partial^{\alpha_{d}}_{x_{d}}}, \quad x=(x_{1}, x_{2}, \cdots, x_{d}).
\]
The fractional Laplacian $\Lambda^{\gamma}=(\sqrt{-\Delta})^{\gamma}$ in $\mathbb{R}^3$ is defined by the Fourier transform representation
\[
\widehat{\Lambda^{\gamma} f}(\xi)=|\xi|^{\gamma}\widehat{f}(\xi).
\]

For $k\in \mathbb{N}$,  we define a Sobolev space $H^{k}$ equipped with the following norm
\[
\|f\|^{2}_{H^{k}}=\|f\|^{2}_{L^{2}}+\sum^{k}_{l=1}\left\|\nabla^{l} f\right\|^{2}_{L^{2}}. 
\]
When $\alpha\in \mathbb{R}_{+}$, we define the $H^{\alpha}$ norm as
\[
\|f\|^{2}_{H^{\alpha}}=\|f\|^{2}_{L^{2}}+\|f\|^{2}_{\dot{H}^{\alpha}}, \quad \|f\|_{\dot{H}^{\alpha}}=\left\|\Lambda^{\alpha}f\right\|_{L^{2}}.
\]
In the energy spaces, we have the following interpolations: for $s_{0}<s<s_{1}$
\eqn \label{energy interpolation}
\left\|f\right\|_{\dot{H}^{s}}\leq \left\|f\right\|^{\theta}_{\dot{H}^{s_{0}}}\left\|f\right\|^{1-\theta}_{\dot{H}^{s_{1}}}, \quad s=\theta s_{0}+(1-\theta)s_{1}.
\een
We also use the following inequalities \cite[Corollary 5.2]{Li19} handling the product of two functions. Let $k\in\mathbb{N}$, $1<p<\infty$, and $1<p_{1},p_{2},p_{3},p_{4}\leq\infty$ satisfy $\frac{1}{p}=\frac{1}{p_{1}}+\frac{1}{p_{2}}=\frac{1}{p_{3}}+\frac{1}{p_{4}}$. Then we have
\begin{subequations}
\begin{align}
\norm{\nabla^{k}(fg)}{L^{p}}&\leq C\left(\norm{\nabla^{k}f}{L^{p_{1}}}\norm{g}{L^{p_{2}}}+\norm{f}{L^{p_{3}}}\norm{\nabla^{k}g}{L^{p_{4}}}\right), \label{eq:product-estimate}
\\
\label{eq:commutator-estimate}
\norm{\nabla^{k}(fg)-f(\nabla^{k}g)}{L^{p}}&\leq C\left(\norm{\nabla^{k}f}{L^{p_{1}}}\norm{g}{L^{p_{2}}}+\norm{\nabla f}{L^{p_{3}}}\norm{\nabla^{k-1}g}{L^{p_{4}}}\right).
\end{align}
\end{subequations}
In the proof of Proposition \ref{Prop perturbed gwp} and Proposition \ref{Prop perturbed decay}, we use $W^{k,\infty}$, $k\in \mathbb{N}$, with the following norm
\[
\|f\|_{W^{k,\infty}}=\|f\|_{L^{\infty}}+ \sum^{k}_{l=1}\|\nabla^{l}f\|_{L^{\infty}}.
\]

\subsection{Useful Lemmas} To prove Proposition \ref{Prop perturbed gwp} in this section and Theorem \ref{Theorem 2.3} in Section \ref{sec:5}, we use the following Gr\"onwall type lemma.

\begin{lemma}\label{Inequality Lemma 2} \cite[Lemma A.4]{Danchin Tan 3}\upshape
Let $X$, $D$, $W$, and $E$ be nonnegative measurable functions such that $X$ is also differentiable. Assume  there exist two non-negative numbers $C_0$ and $\lambda>0$ such that
\[
\frac{d}{dt}X+D\leq WX+C_0 X^{\lambda}D+E
\]
on $[0,T]$. If 
\[
\left(X(0)+\int^T_0 E(\tau)\,d\tau\right)\exp{\left(\int^T_0 W(\tau)\,d\tau\right)}<\left(\frac{1}{2C_0}\right)^{1/\lambda},
\]
then for any $t\in[0,T]$, the following inequality holds
\[
X(t)+\frac{1}{2}\int^{t}_{0}D(\tau)\,d\tau \leq \left(X(0)+\int^t_0 E(\tau)\,d\tau\right)\exp{\left(\int^t_0 W(\tau)\,d\tau\right)}.
\]
\end{lemma}

We also need the following two lemmas in the proof of Proposition \ref{Prop perturbed decay}.

\begin{lemma}\upshape \label{Decay lemma} \cite[Page 134]{Qin}
Let $f$ be a nonnegative differentiable function satisfying the following inequality:
\[
f'+K_{0}f^{1+1/k}\leq \frac{K_{1}}{(1+t)^{1+k}}, \quad f(0)=f_{0},
\]
where $K_0,K_1>0$ are constants and $k>1$. Then, there is a constant $K_2$ such that  
\[
f(t) \leq \frac{K_{2} (k f_{0} +2K_{1})}{(1+t)^k} \quad \text{for all $t\geq0$.}
\]
\end{lemma}

\begin{lemma}\upshape \label{log-ineq lemma}\upshape
\cite[$H^{s}$ version of (2.1)]{Kozono}
For $f\in H^s(\mathbb{R}^3)$, $s>\frac{3}{2}$
\[
\left\|f\right\|_{L^{\infty}} \leq C \left\|f\right\|_{\dot{H}^{\frac{3}{2}}} \left[1+\left(\ln{ \frac{\|f\|_{H^{s}}}{\left\|f\right\|_{\dot{H}^{\frac{3}{2}}}}}\right)^{1/2}\right].
\]
\end{lemma}

\subsection{Proof of Proposition \ref{Prop perturbed gwp}}
We recall (\ref{eq:perturbed-HMHD}):
\eqn \label{eq:4.2}
\begin{split}
&v_{t}-\nu\Delta v+(v\cdot\nabla) v-(b\cdot\nabla) b+\nabla q +(U\cdot\nabla)v-(M\cdot\nabla)b=-(v\cdot\nabla)\overline{v}-(\overline{v}\cdot\nabla) v+(b\cdot\nabla)\overline{b}+(\overline{b}\cdot\nabla) b,\\
&b_{t}-\nu\Delta b+(v\cdot\nabla) b-(b\cdot\nabla) v +\nabla\times((\nabla\times b)\times b)+(U\cdot\nabla)b-(M\cdot\nabla)v+\nabla\times((\nabla\times b)\times M) \\
&=-(v\cdot\nabla)\overline{b}-(\overline{v}\cdot\nabla) b+(b\cdot\nabla)\overline{v}+(\overline{b}\cdot\nabla) v -\nabla\times((\nabla\times b)\times\overline{b})-\nabla\times((\nabla\times\overline{b})\times b),\\
&\dv  v=\dv b=0,\\
&(v_{0},b_{0})=(u_{0}-\overline{u}_{0},B_{0}-\overline{B}_{0}).
\end{split}
\een

Since the linear terms with constant $(U,M)$ do not affect the energy estimates, we can follow \cite{Chae Degond Liu} to the left-hand side of (\ref{eq:4.2}) with the additional terms from the right-hand side of (\ref{eq:4.2}). We first estimate the $L^2$ bound 
\[
\begin{split}
&\frac{1}{2}\frac{d}{dt}\left(\|v\|_{L^2}^2+\|b\|_{L^2}^2\right)+\nu\left(\|\nabla v\|_{L^2}^2+\|\nabla b\|_{L^2}^2\right) \\
&= -\int(v\cdot\nabla)\overline{v}\cdot v+\int(b\cdot\nabla)\overline{b}\cdot v -\int(v\cdot\nabla)\overline{b}\cdot b+\int(b\cdot\nabla)\overline{v}\cdot b -\int((\nabla\times\overline{b})\times b)\cdot(\nabla\times b) \\
&= \int(v\otimes \overline{v}-b\otimes\overline{b}):\nabla v+\int(v\otimes\overline{b}-b\otimes\overline{v}):\nabla b-\int((\nabla\times\overline{b})\times b)\cdot(\nabla\times b) \\
&\leq \left(\|\overline{v}\|_{L^{\infty}}\|v\|_{L^2}+\|\overline{b}\|_{L^{\infty}}\|b\|_{L^2}\right)\|\nabla v\|_{L^2}+\left(\|\overline{b}\|_{L^{\infty}}\|v\|_{L^2}+\|\overline{v}\|_{L^{\infty}}\|b\|_{L^2}\right)\|\nabla b\|_{L^2}+\|\nabla \overline{b}\|_{L^{\infty}}\|b\|_{L^2}\|\nabla b\|_{L^2} \\
&\leq \frac{C}{\nu}\left(\|\overline{v}\|_{L^{\infty}}^2+\|\overline{b}\|_{W^{1,\infty}}^2\right)\left(\|v\|_{L^2}^2+\|b\|_{L^2}^2\right)+\frac{\nu}{2}\left(\|\nabla v\|_{L^2}^2+\|\nabla b\|_{L^2}^2\right).
\end{split}
\]
Hence, we have
\eqn \label{perturbed L2}
\frac{d}{dt}\left(\|v\|_{L^2}^2+\|b\|_{L^2}^2\right)+\nu\left(\|\nabla v\|_{L^2}^2+\|\nabla b\|_{L^2}^2\right) \leq \frac{C}{\nu}\left(\|\overline{v}\|_{L^{\infty}}^2+\|\overline{b}\|_{W^{1,\infty}}^2\right)\left(\|v\|_{L^2}^2+\|b\|_{L^2}^2\right),
\een
which implies \eqref{perturbed weak-ineq}.

We now estimate the $\dot{H}^3$ bound of $(v,b)$
\[
\begin{split}
&\frac{1}{2}\frac{d}{dt} \left(\left\|v\right\|^{2}_{\dot{H}^{3}}+\left\|b\right\|^{2}_{\dot{H}^{3}}\right)+\nu\left(\left\|\nabla v\right\|^{2}_{\dot{H}^{3}}+\left\|\nabla b\right\|^{2}_{\dot{H}^{3}}\right)\\
&=\int \nabla^{3} (v\otimes v-b\otimes b) \cdot \nabla^{3}\nabla v+\int \nabla^{3}(v\otimes b-b\otimes  v) \cdot \nabla^{3}\nabla b -\int\nabla^{3}((\nabla \times b)\times b)\cdot \nabla^{3}(\nabla \times b)\\
&\quad+\int\nabla^{3}(v\otimes\overline{v}+\overline{v}\otimes v-b\otimes\overline{b}-\overline{b}\otimes b):\nabla^{3}\nabla v+\int \nabla^{3}(v\otimes\overline{b}+\overline{v}\otimes b-b\otimes\overline{v}-\overline{b}\otimes v): \nabla^{3}\nabla b \\
&\quad-\int \nabla^{3}((\nabla\times\overline{b})\times b)\cdot \nabla^{3}(\nabla\times b) - \int \nabla^{3}((\nabla\times b)\times\overline{b}) \cdot \nabla^{3}(\nabla\times b) \\
&=\text{I}_{(1)}+\text{I}_{(2)}+\text{I}_{(3)}+\text{I}_{(4)}+\text{I}_{(5)}+\text{I}_{(6)}+\text{I}_{(7)}.
\end{split}
\]
We begin with $\text{I}_{(1)}+\text{I}_{(2)}$ and apply \eqref{eq:product-estimate} to obtain
\[
\begin{split}
\text{I}_{(1)}+\text{I}_{(2)}&\leq C\left(\|v\|_{L^{3}} +\|b\|_{L^{3}}\right)\left(\left\|\nabla^{3}v\right\|_{L^{6}}+\left\|\nabla^{3}b\right\|_{L^{6}}\right) \left(\left\|\nabla^4 v\right\|_{L^2}+\left\|\nabla^4 b\right\|_{L^2}\right)\\
&\leq C\left(\|v\|_{H^3}+\|b\|_{H^3}\right)\left(\left\|\nabla v\right\|_{\dot{H}^{3}}^2+\left\|\nabla b\right\|_{\dot{H}^{3}}^2\right).
\end{split}
\]
We also deal with $\text{I}_{(3)}$ using \eqref{eq:commutator-estimate}
\[
\begin{aligned}
\text{I}_{(3)} &=-\int\left[\nabla^{3}((\nabla \times b)\times b)-\nabla^3(\nabla\times b)\times b\right]\cdot \nabla^{3}(\nabla \times b)  \\
&\leq C\|\nabla b\|_{L^3}\|\nabla^3 b\|_{L^6}\|\nabla^4 b\|_{L^2} \leq C\|b\|_{H^3}\left\|\nabla b\right\|_{\dot{H}^{3}}^2.
\end{aligned}
\]
Using \eqref{eq:product-estimate}, we also bound $\text{I}_{(4)}+\text{I}_{(5)}+\text{I}_{(6)}$ as follows
\[
\begin{split}
\text{I}_{(4)}+\text{I}_{(5)}+\text{I}_{(6)}  &\leq C\left(\|\overline{v}\|_{W^{3,\infty}}+\|\overline{b}\|_{W^{4,\infty}}\right)\left(\|v\|_{H^3}+\|b\|_{H^3}\right)\left(\|\nabla v\|_{\dot{H}^3}+\|\nabla b\|_{\dot{H}^3}\right) \\
&\leq \frac{C}{\nu}\left(\norm{\overline{v}}{W^{3,\infty}}^2+\norm{\overline{b}}{W^{4,\infty}}^2\right)\left(\left\|v\right\|^{2}_{H^{3}}+\left\|b\right\|^{2}_{H^{3}}\right)+\frac{\nu}{4}\left(\|\nabla v\|_{\dot{H}^3}^2+\|\nabla b\|_{\dot{H}^3}^2\right).
\end{split}
\]
As $\text{I}_{(3)}$, we apply \eqref{eq:commutator-estimate} to obtain
\[
\begin{split}
\text{I}_{(7)} &=- \int \left[\nabla^{3}((\nabla\times b)\times\overline{b})-\nabla^3(\nabla\times b)\times \overline{b}\right] \cdot \nabla^{3}(\nabla\times b) \\
&\leq C\|\overline{b}\|_{W^{3,\infty}}\|b\|_{H^3}\|\nabla b\|_{\dot{H}^3}\leq \frac{C}{\nu}\norm{\overline{b}}{W^{4,\infty}}^2\left\|b\right\|^{2}_{H^{3}}+\frac{\nu}{4}\|\nabla b\|_{\dot{H}^3}^2.
\end{split}
\]
In sum, we obtain
\eqn \label{perturbed H3dot}
\begin{split}
&\frac{d}{dt} \left(\left\|v\right\|^{2}_{\dot{H}^{3}}+\left\|b\right\|^{2}_{\dot{H}^{3}}\right)+\nu\left(\left\|\nabla v\right\|^{2}_{\dot{H}^{3}}+\left\|\nabla v\right\|^{2}_{\dot{H}^{3}}\right)\\
&\leq C\left(\|v\|_{H^3}+\|b\|_{H^3}\right)\left(\left\|\nabla v\right\|_{\dot{H}^{3}}^2+\left\|\nabla b\right\|_{\dot{H}^{3}}^2\right) +\frac{C}{\nu}\left(\norm{\overline{v}}{W^{3,\infty}}^2+\norm{\overline{b}}{W^{4,\infty}}^2\right)\left(\left\|v\right\|^{2}_{H^{3}}+\left\|b\right\|^{2}_{H^{3}}\right).
\end{split}
\een

Combining \eqref{perturbed L2} and \eqref{perturbed H3dot}, we arrive at
\[
\begin{split}
&\frac{d}{dt} \left(\left\|v\right\|^{2}_{{H}^{3}}+\left\|b\right\|^{2}_{{H}^{3}}\right)+\nu\left(\left\|\nabla v\right\|^{2}_{{H}^{3}}+\left\|\nabla v\right\|^{2}_{{H}^{3}}\right)\\
&\leq C_1\left(\left\|v\right\|^{2}_{H^{3}}+\left\|b\right\|^{2}_{H^{3}}\right)^{\frac{1}{2}} \left(\left\|\nabla v\right\|^{2}_{{H}^{3}}+\left\|\nabla b\right\|^{2}_{{H}^{3}}\right)+\frac{C_2}{\nu}\left(\norm{\overline{v}}{W^{3,\infty}}^2+\norm{\overline{b}}{W^{4,\infty}}^2\right)\left(\left\|v\right\|^{2}_{H^{3}}+\left\|b\right\|^{2}_{H^{3}}\right)
\end{split}
\]
for some $C_1,C_2>0$.
By Lemma \ref{Inequality Lemma 2}, we derive the following bound for all $t>0$
\[
\begin{split}
&\left\|v(t)\right\|^{2}_{H^{3}}+\left\|b(t)\right\|^{2}_{H^{3}}+\frac{\nu}{2}\int^{t}_{0}\left(\left\|\nabla v(\tau)\right\|^{2}_{H^{3}}+\left\|\nabla v(\tau)\right\|^{2}_{H^{3}}\right)d\tau \\
&\leq \left(\left\|v_{0}\right\|^{2}_{H^{3}}+\left\|b_{0}\right\|^{2}_{H^{3}}\right)\exp\left[\frac{C_2}{\nu}\int^{t}_{0}\left(\norm{\overline{v}(\tau)}{W^{3,\infty}}^2+\norm{\overline{b}(\tau)}{W^{4,\infty}}^2\right)d\tau\right] \leq \left(\left\|v_{0}\right\|^{2}_{H^{3}}+\left\|b_{0}\right\|^{2}_{H^{3}}\right)\exp{\left[\frac{C_0C_2}{\nu}\right]}
\end{split}
\]
when 
\[
\left(\left\|v_{0}\right\|^{2}_{H^{3}}+\left\|b_{0}\right\|^{2}_{H^{3}}\right)\exp{\left[\frac{C_0C_2}{\nu}\right]}<\frac{\nu^2}{4C_1^2}.
\]
This completes the proof of Proposition \ref{Prop perturbed gwp}.

\subsection{Proof of Proposition \ref{Prop perturbed decay}}
From \eqref{perturbed H3dot}, \eqref{perturbed eq bounds}, and \eqref{perturbed decay condition}, we have
\[
\frac{d}{dt} \left(\left\|v\right\|^{2}_{\dot{H}^{3}}+\left\|b\right\|^{2}_{\dot{H}^{3}}\right)+\frac{\nu}{2}\left(\left\|\nabla v\right\|^{2}_{\dot{H}^{3}}+\left\|\nabla b\right\|^{2}_{\dot{H}^{3}}\right) \leq \frac{C\epsilon}{\nu}\left(\norm{\overline{v}}{W^{3,\infty}}^{2}+\norm{\overline{b}}{W^{4,\infty}}^{2}\right)\leq \frac{K_{1}}{(1+t)^{4}}
\]
for $K_1={C\epsilon}{\nu}^{-1}$.
Since 
\[
\left\|v\right\|^{2}_{\dot{H}^{3}}+\left\|b\right\|^{2}_{\dot{H}^{3}} \leq \left\|v\right\|^{\frac{1}{2}}_{L^{2}}\left\|\nabla v\right\|^{\frac{3}{2}}_{\dot{H}^{3}}+\left\|b\right\|^{\frac{1}{2}}_{L^{2}}\left\|\nabla b\right\|^{\frac{3}{2}}_{\dot{H}^{3}}\leq C\left(\left\|\nabla v\right\|^{\frac{3}{2}}_{\dot{H}^{3}}+\left\|\nabla b\right\|^{\frac{3}{2}}_{\dot{H}^{3}}\right)
\]
from (\ref{energy interpolation}), we arrive at 
\[
\frac{d}{dt} \left(\left\|v\right\|^{2}_{\dot{H}^{3}}+\left\|b\right\|^{2}_{\dot{H}^{3}}\right)+K_{0} \left(\left\|v\right\|^{2}_{\dot{H}^{3}}+\left\|b\right\|^{2}_{\dot{H}^{3}}\right)^{4/3} \leq \frac{K_{1}}{(1+t)^{4}},
\]
for $K_{0}=\nu C$. By Lemma \ref{Decay lemma}, we obtain 
\[
\left\|v(t)\right\|^{2}_{\dot{H}^{3}}+\left\|b(t)\right\|^{2}_{\dot{H}^{3}} \leq K_{2} \left(3\left(\left\|v_{0}\right\|^{2}_{\dot{H}^{3}}+\left\|b_{0}\right\|^{2}_{\dot{H}^{3}}\right) +2K_{1}\right)(1+t)^{-3}
\]
for all $t>0$. \eqref{perturbed decay1} is derived by using (\ref{energy interpolation}) for $0<k<3$. 

We also deduce from Lemma \ref{log-ineq lemma} and (\ref{perturbed eq bounds}), \eqref{perturbed decay1} with $k=3/2$ that
\[
\left\|v(t)\right\|_{L^{\infty}}+\left\|b(t)\right\|_{L^{\infty}}\leq C(\epsilon) \left(\left\|v(t)\right\|_{\dot{H}^{\frac{3}{2}}}+\left\|b(t)\right\|_{\dot{H}^{\frac{3}{2}}}\right)\leq C(\epsilon) (1+t)^{-\frac{3}{4}},
\]
where $C(\epsilon)$ is a constant only depending on $\epsilon$. This completes the proof of Proposition \ref{Prop perturbed decay}.

\section{Proof of Theorem \ref{Theorem 2.3}} \label{sec:5}
We use the following inequalities repeatedly when we prove Theorem \ref{Theorem 2.3}:
\[
\left\|f\right\|_{L^{3}}\leq C \left\|f\right\|_{\dot{H}^{\frac{1}{2}}}, \quad \left\|f\right\|_{L^{6}}\leq C \left\|f\right\|_{\dot{H}^{1}}.
\]

\subsection{Proof of Theorem \ref{Theorem 2.3}}
In the proof of Theorem \ref{Theorem 2.3}, there are many constants depending on $(\alpha,\beta)$ to which we write $C$ for simplicity. We first recall energy inequality:
\eqn \label{energy inequality}
\left\|u(t)\right\|^{2}_{L^{2}} + \left\|B(t)\right\|^{2}_{L^{2}}+2\nu \int^{t}_{0} \left\|\nabla u(\tau)\right\|^{2}_{L^{2}}d\tau +2\eta \int^{t}_{0} \left\|\nabla B(\tau)\right\|^{2}_{L^{2}}d\tau \leq \left\|u_{0}\right\|^{2}_{L^{2}}+\left\|B_{0}\right\|^{2}_{L^{2}}=\mathcal{E}_{0}.
\een

\vspace{1ex}
\noindent
\textbf{Step 1.} Since the local well-posedness of $H^3$ solutions $(u,B)$ to \eqref{vH MHD} is established in \cite{Chae Degond Liu}, we begin with $(u,B)\in C([0,T];H^3)$ for some $T>0$.
Let $\Phi=B+\omega-\alpha u$ and $\Psi=u-J+\beta B$ and we rewrite \eqref{vH MHD} as follows:
\begin{equation}\label{HMHD 5}
\begin{split}
\Phi_{t}-\mu\Delta\Phi =&-\nabla \times (\Phi\times u)+\alpha\Psi\times B+\alpha \Phi\times u+\alpha\nabla\overline{p}+\frac{\nu-\eta}{2}\Delta\Phi-(\nu-\eta)\Delta B\\
\Psi_{t}-\mu\Delta\Psi =&-\Psi\times B-\Phi\times u-\nabla\overline{p}-\nabla \times (\nabla \times  (\Psi\times B))+\beta\nabla \times  (\Psi\times B) \\
&-\frac{\nu-\eta}{2}\Delta\Psi+(\nu-\eta)\Delta u,
\end{split}
\end{equation}
where $\mu=(\nu+\eta)/2$. From (\ref{HMHD 5}), we first have 
\eqn \label{L2 bound phi psi}
\begin{split}
&\frac{1}{2}\frac{d}{dt}\left(\norm{\Phi}{L^{2}}^{2}+\norm{\Psi}{L^{2}}^{2}\right)+\mu\left(\norm{\nabla\Phi}{L^{2}}^{2}+\norm{\nabla\Psi}{L^{2}}^{2}\right)\\
& =\left[\alpha\int(\Psi\times B)\cdot\Phi-\int(\Phi\times u)\cdot\Psi\right] +\left[-\int(\Phi\times u)\cdot \nabla\times\Phi+\beta\int(\Psi\times B)\cdot\nabla\times\Psi\right] \\
&\quad-\int (\Psi\times B)\cdot \nabla \times \nabla \times \Psi-\frac{\nu-\eta}{2}\left(\norm{\nabla\Phi}{L^2}^2-\norm{\nabla\Psi}{L^2}^2\right)\\
&\quad+(\nu-\eta)\left(\int\nabla B:\nabla\Phi-\int\nabla u:\nabla\Psi\right)=\text{I}_{1}+\text{I}_{2}+\text{I}_{3}+\text{I}_{4}+\text{I}_{5}.
\end{split}
\een
Since 
\[
u=\frac{\Psi-\beta\Phi+J+\beta\omega}{1+\alpha\beta}, \quad B=\frac{\Phi+\alpha\Psi-\omega+\alpha J}{1+\alpha\beta}
\]
 when $1+\alpha\beta\neq0$, we have
 \eqn \label{I 1}
\begin{split}
\text{I}_{1}&=\frac{1}{1+\alpha\beta}\left(\int \left[-\alpha (\Psi\times \omega)\cdot \Phi +\alpha^{2} (\Psi\times J)\cdot \Phi -(\Phi\times J)\cdot \Psi -\beta (\Phi\times \omega)\cdot \Psi\right]\right)\\
&\leq C\left(\norm{\nabla u}{L^{2}}+\norm{\nabla B}{L^{2}}\right)\norm{\Psi}{L^{3}}\norm{\Phi}{L^{6}}\leq C \left(\norm{\nabla u}{L^{2}}^{2}+\norm{\nabla B}{L^{2}}^{2}\right)\norm{\Psi}{\dot{H}^{\frac{1}{2}}}\norm{\nabla\Phi}{L^{2}}\\
&\leq \frac{C}{\mu}\left(\norm{\nabla u}{L^{2}}^{2}+\norm{\nabla B}{L^{2}}^{2}\right)\norm{\Psi}{\dot{H}^{\frac{1}{2}}}^{2}+\frac{\mu}{8}\norm{\nabla\Phi}{L^{2}}^{2}.
\end{split}
\een
We also bound $\text{I}_{2}$ as follows:
\[
\begin{split}
\text{I}_2 &\leq C \left(\norm{u}{L^{6}}+\norm{B}{L^{6}}\right) \left(\norm{\Psi}{L^{3}}+\norm{\Phi}{L^{3}}\right) \left(\norm{\nabla \Psi}{L^{2}}+\norm{\nabla \Phi}{L^{2}}\right)\\
&\leq \frac{C}{\mu} \left(\norm{\nabla u}{L^{2}}^{2}+\norm{\nabla B}{L^{2}}^{2}\right)\left(\norm{\Phi}{\dot{H}^{\frac{1}{2}}}^{2}+\norm{\Psi}{\dot{H}^{\frac{1}{2}}}^{2}\right)+\frac{\mu}{16}\left(\norm{\nabla\Phi}{L^{2}}^{2} +\norm{\nabla\Psi}{L^{2}}^{2}\right). 
\end{split}
\]
Since
\[
\text{I}_{3}=\int (\Psi\times B)\cdot \Delta \Psi=-\int (\partial_{k}\Psi \times B)\cdot \partial_{k} \Psi-\int (\Psi \times \partial_{k}B)\cdot \partial_{k} \Psi=-\int (\Psi \times \partial_{k}B)\cdot \partial_{k} \Psi
\]
and  
\[
\norm{\nabla B}{L^{6}}\leq C \norm{J}{L^{6}}\leq C\left(\norm{B}{L^{6}}+ \norm{u}{L^{6}} +\norm{\Psi}{L^{6}}\right),
\]
we bound $\text{I}_{3}$ as
\[
\begin{split}
\text{I}_{3}&\leq C\left(\norm{B}{L^{6}}+ \norm{u}{L^{6}} +\norm{\Psi}{L^{6}}\right)\norm{\Psi}{L^{3}}\norm{\nabla \Psi}{L^{2}}\\
&\leq \frac{C}{\mu}\left(\norm{\nabla u}{L^{2}}^{2}+\norm{\nabla B}{L^{2}}^{2}\right)\norm{\Psi}{\dot{H}^{\frac{1}{2}}}^{2}+C\norm{\Psi}{\dot{H}^{\frac{1}{2}}}\norm{\nabla \Psi}{L^{2}}^{2} +\frac{\mu}{8} \norm{\nabla \Psi}{L^{2}}^{2}.
\end{split}
\]
We bound the  remaining parts $\text{I}_{4}$ and $\text{I}_{5}$:
\[
\begin{split}
\text{I}_4 &\leq \frac{|\nu-\eta|}{2}\left(\norm{\nabla\Phi}{L^2}^2+\norm{\nabla\Psi}{L^2}^2\right)\leq \frac{\mu}{16}\left(\norm{\nabla\Phi}{L^2}^2+\norm{\nabla\Psi}{L^2}^2\right),\\
\text{I}_5 &\leq \frac{\mu}{16}\left(\norm{\nabla\Phi}{L^2}^2+\norm{\nabla\Psi}{L^2}^2\right)+C\frac{(\nu-\eta)^2}{\nu+\eta}\left(\norm{\nabla u}{L^2}^2+\norm{\nabla B}{L^2}^2\right).
\end{split}
\] 
Therefore, we obtain 
\eqn \label{eq:4.4}
\begin{split}
&\frac{1}{2}\frac{d}{dt}\left(\norm{\Phi}{L^{2}}^{2}+\norm{\Psi}{L^{2}}^{2}\right)+\mu\left(\norm{\nabla\Phi}{L^{2}}^{2}+\norm{\nabla\Psi}{L^{2}}^{2}\right)\\
&\leq \frac{C}{\mu}\left(\norm{\nabla u}{L^{2}}^{2}+\norm{\nabla B}{L^{2}}^{2}\right)\left(\norm{\Phi}{\dot{H}^{\frac{1}{2}}}^{2}+\norm{\Psi}{\dot{H}^{\frac{1}{2}}}^{2}\right)+C\norm{\Psi}{\dot{H}^{\frac{1}{2}}}\norm{\nabla \Psi}{L^{2}}^{2} \\
&\quad +C\frac{(\nu-\eta)^2}{\nu+\eta}\left(\norm{\nabla u}{L^2}^2+\norm{\nabla B}{L^2}^2\right)+\frac{5\mu}{16}\left(\norm{\nabla\Phi}{L^2}^2+\norm{\nabla\Psi}{L^2}^2\right).
\end{split}
\een

\vspace{1ex}

By the $L^{2}$ inner product of (\ref{HMHD 5}) with $(\Lambda\Phi,\Lambda\Psi)$, we also have
\begin{align*}
&\frac{1}{2}\frac{d}{dt}\left(\norm{\Phi}{\dot{H}^{\frac{1}{2}}}^{2}+\norm{\Psi}{\dot{H}^{\frac{1}{2}}}^{2}\right)+\mu\left(\norm{\Phi}{\dot{H}^{\frac{3}{2}}}+\norm{\Psi}{\dot{H}^{\frac{3}{2}}}\right)\\
&=\left[\alpha\int(\Psi\times B)\cdot\Lambda\Phi+\alpha\int(\Phi\times u)\cdot\Lambda\Phi-\int(\Psi\times B)\cdot\Lambda\Psi-\int(\Phi\times u)\cdot\Lambda\Psi\right]\\
&\quad-\int\nabla \times (\Phi\times u)\cdot\Lambda\Phi+\beta \int\nabla \times (\Psi\times B)\cdot\Lambda\Psi-\int\nabla \times (\nabla \times (\Psi\times B))\cdot\Lambda\Psi\\
&\quad-\frac{\nu-\eta}{2}\left(\norm{\Phi}{\dot{H}^{\frac{3}{2}}}^2-\norm{\Psi}{\dot{H}^{\frac{3}{2}}}^2\right)+(\nu-\eta)\left(\int\nabla\times\nabla\times B\cdot \Lambda\Phi-\int\nabla\times\nabla\times u\cdot\Lambda\Psi\right)\\
&=\text{II}_{1}+\text{II}_{2}+\text{II}_{3}+\text{II}_{4}+\text{II}_{5}+\text{II}_{6}.
\end{align*}
We begin with $\text{II}_{1}$:
\begin{align*}
\text{II}_{1}&\leq C\left(\norm{u}{L^{6}}+\norm{B}{L^{6}}\right) \left(\norm{\Phi}{L^{2}}+\norm{\Psi}{L^{2}}\right)\left(\norm{\nabla\Phi}{L^{3}}+\norm{\nabla\Psi}{L^{3}}\right)\\
&\leq \frac{C}{\mu}\left(\norm{\nabla  u}{L^{2}}^{2}+\norm{\nabla B}{L^{2}}^{2}\right) \left(\norm{\Phi}{L^{2}}^{2}+\norm{\Psi}{L^{2}}^{2}\right)+\frac{\mu}{8}\left(\norm{\Phi}{\dot{H}^{\frac{3}{2}}}^{2} +\norm{\Psi}{\dot{H}^{\frac{3}{2}}}^{2}\right).
\end{align*}
We now bound $\text{II}_{2}$:
\[
\begin{split}
\text{II}_{2}&\leq C\norm{\nabla u}{L^{2}}\norm{\Phi}{L^{6}}\norm{\nabla\Phi}{L^{3}}+C\norm{u}{L^{6}}\norm{\nabla\Phi}{L^{2}}\norm{\nabla\Phi}{L^{3}}\\
&\leq C\norm{\nabla u}{L^{2}}\norm{\nabla \Phi}{L^{2}}\norm{\Phi}{\dot{H}^{\frac{3}{2}}}\leq C\norm{\nabla u}{L^{2}}\norm{\Phi}{\dot{H}^{\frac{1}{2}}}^{\frac{1}{2}}\norm{\Phi}{\dot{H}^{\frac{3}{2}}}^{\frac{3}{2}}.
\end{split}
\]
Using $\Phi=B+\omega-\alpha u$,
\[
\norm{\nabla u}{L^{2}} \leq \norm{\nabla u}{L^{2}}^{\frac{1}{2}}\norm{\omega}{L^{2}}^{\frac{1}{2}}\leq C\norm{\nabla u}{L^{2}}^{\frac{1}{2}} \left(\norm{B}{L^{2}}^{\frac{1}{2}}+\norm{\Phi}{L^{2}}^{\frac{1}{2}}+\sqrt{|\alpha|}\norm{u}{L^{2}}^{\frac{1}{2}}\right)
\]
and so we have
\[
\begin{split}
\text{II}_{2}&\leq C\norm{\nabla u}{L^{2}}^{\frac{1}{2}} \left(\norm{B}{L^{2}}^{\frac{1}{2}}+\norm{u}{L^{2}}^{\frac{1}{2}}\right)\norm{\Phi}{\dot{H}^{\frac{1}{2}}}^{\frac{1}{2}}\norm{\Phi}{\dot{H}^{\frac{3}{2}}}^{\frac{3}{2}}+C\norm{\nabla u}{L^{2}}^{\frac{1}{2}} \norm{\Phi}{L^{2}}^{\frac{1}{2}}\norm{\Phi}{\dot{H}^{\frac{1}{2}}}^{\frac{1}{2}}\norm{\Phi}{\dot{H}^{\frac{3}{2}}}^{\frac{3}{2}}\\
& \leq \frac{C}{\mu^{3}}\left(\norm{u}{L^{2}}^{2}+\norm{B}{L^{2}}^{2}\right)\norm{\nabla u}{L^{2}}^{2} \norm{\Phi}{\dot{H}^{\frac{1}{2}}}^{2} + \frac{C}{\mu}\norm{\nabla u}{L^{2}}^{2} \norm{\Phi}{L^{2}}^{2} +{\frac{\mu}{8}}\norm{\Phi}{\dot{H}^{\frac{3}{2}}}^{2} +\mu^{\frac{1}{3}}\norm{\Phi}{\dot{H}^{\frac{1}{2}}}^{\frac{2}{3}}\norm{\Phi}{\dot{H}^{\frac{3}{2}}}^{2} \\
& \leq \frac{C}{\mu^{3}}\left(\norm{u}{L^{2}}^{2}+\norm{B}{L^{2}}^{2}\right)\norm{\nabla u}{L^{2}}^{2} \norm{\Phi}{\dot{H}^{\frac{1}{2}}}^{2} + \frac{C}{\mu}\norm{\nabla u}{L^{2}}^{2} \norm{\Phi}{L^{2}}^{2} +{\frac{\mu}{4}}\norm{\Phi}{\dot{H}^{\frac{3}{2}}}^{2}+C\norm{\Phi}{\dot{H}^{\frac{1}{2}}}\norm{\Phi}{\dot{H}^{\frac{3}{2}}}^2 .
\end{split}
\]
Similarly, we use $\Psi=u-J+\beta B$ to obtain
\[
\begin{split}
\text{II}_{3}&\leq C\norm{\nabla B}{L^2}\norm{\nabla\Psi}{L^2}\norm{\Psi}{\dot{H}^{\frac{3}{2}}}\leq C\norm{\nabla B}{L^2}\norm{\Psi}{\dot{H}^{\frac{1}{2}}}^{\frac{1}{2}}\norm{\Psi}{\dot{H}^{\frac{3}{2}}}^{\frac{3}{2}} \\
&\leq \frac{C}{\mu^{3}}\left(\norm{u}{L^{2}}^{2}+\norm{B}{L^{2}}^{2}\right)\norm{\nabla B}{L^{2}}^{2} \norm{\Psi}{\dot{H}^{\frac{1}{2}}}^{2}+\frac{C}{\mu}\norm{\nabla B}{L^{2}}^{2} \norm{\Psi}{L^{2}}^{2}  +{\frac{\mu}{8}}\norm{\Psi}{\dot{H}^{\frac{3}{2}}}^{2}+C\norm{\Psi}{\dot{H}^{\frac{1}{2}}}\norm{\Psi}{\dot{H}^{\frac{3}{2}}}^{2}.
\end{split}
\]
We now bound $\text{II}_{4}$:
\begin{align*}
\text{II}_{4}&=\int (\Psi\times B)\cdot \Delta \Lambda \Psi =-\int (\Psi\times B)\cdot \Lambda^{3} \Psi=-\int\left[\Lambda^{\frac{3}{2}}(\Psi\times B)-\Lambda^{\frac{3}{2}}\Psi\times B\right]\cdot\Lambda^{\frac{3}{2}}\Psi\\
&\leq C\left(\norm{\Lambda^{\frac{3}{2}}B}{L^{3}}\norm{\Psi}{L^{6}}+\norm{\nabla B}{L^{6}}\norm{\Lambda^{\frac{1}{2}}\Psi}{L^{3}}\right)\norm{\Lambda^{\frac{3}{2}}\Psi}{L^{2}}\leq C\norm{\nabla J}{L^{2}}\norm{\nabla\Psi}{L^{2}}\norm{\Psi}{\dot{H}^{\frac{3}{2}}}.
\end{align*}
Since 
\[
\norm{\nabla J}{L^{2}} \leq \norm{\nabla \Psi}{L^{2}}+\norm{\nabla u}{L^{2}}+|\beta|\norm{\nabla B}{L^{2}},
\]
we bound $\text{II}_{4}$ as follows:
\begin{align*}
\text{II}_{4}&\leq C\norm{\nabla\Psi}{L^{2}}^{2}\norm{\Psi}{\dot{H}^{\frac{3}{2}}}+C\left(\norm{\nabla u}{L^{2}}+\norm{\nabla B}{L^{2}}\right)\norm{\nabla\Psi}{L^{2}}\norm{\Psi}{\dot{H}^{\frac{3}{2}}}\\
&\leq C\norm{\Psi}{\dot{H}^{\frac{1}{2}}}\norm{\Psi}{\dot{H}^{\frac{3}{2}}}^{2}+C\left(\norm{\nabla u}{L^{2}}+\norm{\nabla B}{L^{2}}\right)\norm{\Psi}{\dot{H}^{\frac{1}{2}}}^{\frac{1}{2}}\norm{\Psi}{\dot{H}^{\frac{3}{2}}}^{\frac{3}{2}}.
\end{align*}
Here, the second term of the right-hand side is treated as in $\text{II}_3$:
\[
\begin{split}
& C\left(\norm{\nabla u}{L^{2}}+\norm{\nabla B}{L^{2}}\right)\norm{\Psi}{\dot{H}^{\frac{1}{2}}}^{\frac{1}{2}}\norm{\Psi}{\dot{H}^{\frac{3}{2}}}^{\frac{3}{2}} \\
&\leq \frac{C}{\mu^{3}}\left(\norm{u}{L^{2}}^{2}+\norm{B}{L^{2}}^{2}\right)\left(\norm{\nabla u}{L^2}^2+\norm{\nabla B}{L^2}^2\right)\norm{\Psi}{\dot{H}^{\frac{1}{2}}}^{2}+\frac{C}{\mu}\left(\norm{\nabla u}{L^2}^2+\norm{\nabla B}{L^2}^2\right)\norm{\Psi}{L^{2}}^{2} \\
&\quad+{\frac{\mu}{8}}\norm{\Psi}{\dot{H}^{\frac{3}{2}}}^{2}+C\norm{\Psi}{\dot{H}^{\frac{1}{2}}}\norm{\Psi}{\dot{H}^{\frac{3}{2}}}^{2}.
\end{split} 
\]
We now bound $\text{II}_{5}$:
\[
\text{II}_{5}\leq \frac{|\nu-\eta|}{2}\left(\norm{\Phi}{\dot{H}^{\frac{3}{2}}}^2+\norm{\Psi}{\dot{H}^{\frac{3}{2}}}^2\right)\leq \frac{\mu}{8}\left(\norm{\Phi}{\dot{H}^{\frac{3}{2}}}^2+\norm{\Psi}{\dot{H}^{\frac{3}{2}}}^2\right).
\] 
We finally bound $\text{II}_{6}$. Using $\Phi=B+\omega-\alpha u$ and $\Psi=u-J+\beta B$, 
\[
\begin{split}
\text{II}_6 &=(\nu-\eta)\left(\int\nabla\times(u+\beta B-\Psi)\cdot\Lambda\Phi-\int\nabla\times(\alpha u-B+\Phi)\cdot\Lambda\Psi\right) \\
&\leq |\nu-\eta|\left(\norm{\nabla\Phi}{L^2}^2+\norm{\nabla\Psi}{L^2}^2\right)+C|\nu-\eta|\left(\norm{\nabla u}{L^2}+\norm{\nabla B}{L^2}\right)\left(\norm{\nabla\Phi}{L^2}+\norm{\nabla\Psi}{L^2}\right)\\
&\leq \frac{3\mu}{16}\left(\norm{\nabla\Phi}{L^2}^2+\norm{\nabla\Psi}{L^2}^2\right)+C\frac{(\nu-\eta)^2}{\nu+\eta}\left(\norm{\nabla u}{L^2}^2+\norm{\nabla B}{L^2}^2\right).
\end{split}
\]
In sum, we obtain 
\eqn \label{eq:4.5}
\begin{split}
&\frac{1}{2}\frac{d}{dt}\left(\norm{\Phi}{\dot{H}^{\frac{1}{2}}}^{2}+\norm{\Psi}{\dot{H}^{\frac{1}{2}}}^{2}\right)+\mu\left(\norm{\Phi}{\dot{H}^{\frac{3}{2}}}+\norm{\Psi}{\dot{H}^{\frac{3}{2}}}\right)\\
&\leq \frac{C}{\mu}\left(\norm{\nabla  u}{L^{2}}^{2}+\norm{\nabla B}{L^{2}}^{2}\right) \left(\norm{\Phi}{L^{2}}^{2}+\norm{\Psi}{L^{2}}^{2}\right)\\
&\quad+\frac{C}{\mu^{3}} \left(\norm{u}{L^{2}}^{2}+\norm{B}{L^{2}}^{2}\right)\left(\norm{\nabla u}{L^{2}}^{2}+\norm{\nabla B}{L^{2}}^{2}\right) \left(\norm{\Phi}{\dot{H}^{\frac{1}{2}}}^{2}+\norm{\Psi}{\dot{H}^{\frac{1}{2}}}^{2}\right) \\
&\quad+C\frac{(\nu-\eta)^2}{\nu+\eta}\left(\norm{\nabla u}{L^2}^2+\norm{\nabla B}{L^2}^2\right) +\frac{3\mu}{16}\left(\norm{\nabla\Phi}{L^2}^2+\norm{\nabla\Psi}{L^2}^2\right)\\
&\quad+\frac{\mu}{2}\left(\norm{\Phi}{\dot{H}^{\frac{3}{2}}}^2+\norm{\Psi}{\dot{H}^{\frac{3}{2}}}^2\right)+C\norm{\Phi}{\dot{H}^{\frac{1}{2}}}\norm{\Phi}{\dot{H}^{\frac{1}{2}}}^2+C\norm{\Psi}{\dot{H}^{\frac{1}{2}}}\norm{\Psi}{\dot{H}^{\frac{3}{2}}}^2.
\end{split}
\een

\vspace{1ex}

By (\ref{eq:4.4}) and (\ref{eq:4.5}), we arrive at
\eqn \label{eq:4.7}
\begin{split}
&\frac{d}{dt}\left(\norm{\Phi}{H^{\frac{1}{2}}}^{2}+\norm{\Psi}{H^{\frac{1}{2}}}^{2}\right)+\mu\left(\norm{\nabla \Phi}{H^{\frac{1}{2}}}^{2}+\norm{\nabla \Psi}{H^{\frac{1}{2}}}^{2}\right)\\
&\leq \frac{C_1}{2}\left(\frac{1}{\mu}+\frac{\mathcal{E}_0}{\mu^3}\right)\left(\norm{\nabla u}{L^2}^2+\norm{\nabla B}{L^2}^2\right)\left(\norm{\Phi}{H^{\frac{1}{2}}}^{2}+\norm{\Psi}{H^{\frac{1}{2}}}^{2}\right) \\
&\quad+\frac{C_2}{2}\left(\norm{\Phi}{H^{\frac{1}{2}}}^{2}+\norm{\Psi}{H^{\frac{1}{2}}}^{2}\right)^{\frac{1}{2}}\left(\norm{\nabla\Phi}{H^{\frac{1}{2}}}^{2}+\norm{\nabla\Psi}{H^{\frac{1}{2}}}^{2}\right) + \frac{C_3}{4}\frac{(\nu-\eta)^2}{\nu+\eta}\left(\norm{\nabla u}{L^2}^2+\norm{\nabla B}{L^2}^2\right).
\end{split}
\een
Set
\[
\begin{aligned}
&X=\norm{\Phi}{H^{\frac{1}{2}}}^{2}+\norm{\Psi}{H^{\frac{1}{2}}}^{2}, \quad D=\mu \left(\norm{\nabla \Phi}{H^{\frac{1}{2}}}^{2}+\norm{\nabla \Psi}{H^{\frac{1}{2}}}^{2}\right),\quad C_0=\frac{C_2}{2\mu},\quad \lambda=\frac{1}{2}, \\
& W=\frac{C_1}{2}\left(\frac{1}{\mu}+\frac{\mathcal{E}_0}{\mu^3}\right)\left(\norm{\nabla u}{L^2}^2+\norm{\nabla B}{L^2}^2\right),\quad E=\frac{C_3}{4}\frac{(\nu-\eta)^2}{\nu+\eta}\left(\norm{\nabla u}{L^2}^2+\norm{\nabla B}{L^2}^2\right).
\end{aligned}
\] 
By Lemma \ref{Inequality Lemma 2}, we derive that if $(u_0,B_0)$ satisfies
\[
\left(\norm{\Phi_{0}}{H^{\frac{1}{2}}}^{2}+\norm{\Psi_{0}}{H^{\frac{1}{2}}}^{2}+C_3\mathcal{E}_0\frac{(\nu-\eta)^2}{(\nu+\eta)^2}\right)\exp{\left(\frac{C_1\mathcal{E}_0}{\mu^2}+\frac{C_1\mathcal{E}_0^2}{\mu^4}\right)}<\frac{\mu^2}{C_2^2},
\]
then, the solutions $(u,B)$ satisfies
\eqn \label{smallness solution HMHD dd}
\norm{\Phi(t)}{H^{\frac{1}{2}}}^{2}+\norm{\Psi(t)}{H^{\frac{1}{2}}}^{2}+\frac{\mu}{2}\int^{t}_{0} \left(\norm{\nabla \Phi(\tau)}{H^{\frac{1}{2}}}^{2}+\norm{\nabla \Psi(\tau)}{H^{\frac{1}{2}}}^{2}\right)d\tau <\frac{\mu^2}{C_2^2}
\een
for all $t\in [0,T]$.

\vspace{1ex}
\noindent
\textbf{Step 2.}
Now we let the maximal existence time ${T_{\text{max}}}>0$ and suppose $0<{T_{\text{max}}}<\infty$. By using \eqref{energy inequality} and \eqref{smallness solution HMHD dd}, we estimate $\omega=\alpha u-B+\Phi$ and $J=u+\beta B-\Psi$ as
\eqn \label{omega J L2 bound dd}
\begin{split}
&\left\|\omega(t)\right\|^{2}_{L^{2}} +\left\|J(t)\right\|^{2}_{L^{2}}+\frac{\mu}{2}\int^{t}_{0}\left(\left\|\nabla \omega(\tau)\right\|^{2}_{L^{2}} +\left\|\nabla J(\tau)\right\|^{2}_{L^{2}}\right)d\tau\\
& \leq C\left\|u(t)\right\|^{2}_{L^{2}}+C\left\|B(t)\right\|^{2}_{L^{2}} +C\int^{t}_{0}\left(\nu\left\|\nabla u(t)\right\|^{2}_{L^{2}}+\eta\left\|\nabla B(t)\right\|^{2}_{L^{2}}\right)d\tau\\
&\quad+\left\|\Phi(t)\right\|^{2}_{L^{2}} +\left\|\Psi(t)\right\|^{2}_{L^{2}}+\frac{\mu}{2}\int^{t}_{0}\left( \left\|\nabla \Phi(t)\right\|^{2}_{L^{2}} +\left\|\nabla \Psi(t)\right\|^{2}_{L^{2}}\right)d\tau \leq C\mathcal{E}_0+\frac{\mu^2}{C_2^2}
\end{split}
\een
for all $t\in[0,{T_{\text{max}}})$.
In particular, (\ref{energy inequality}) and (\ref{omega J L2 bound dd}) imply
\[
\mu\int^{t}_{0} \left\|u(\tau)\right\|^{2}_{\dot{H}^{\frac{3}{2}}}d\tau \leq C\mu\int^{t}_{0} \left(\left\|\nabla u(\tau)\right\|^{2}_{L^{2}}+\left\|\nabla \omega(\tau)\right\|^{2}_{L^{2}}\right)d\tau \leq C\left(\mathcal{E}_0+\frac{\mu^2}{C_2^2}\right).
\]
Similarly, we bound $J$ using \eqref{energy inequality}, \eqref{smallness solution HMHD dd} and \eqref{omega J L2 bound dd} 
\[
\begin{split}
\left\|J(t)\right\|^{2}_{\dot{H}^{\frac{1}{2}}} +\frac{\mu}{2}\int^{t}_{0}\left\|J(\tau)\right\|^{2}_{\dot{H}^{\frac{3}{2}}}d\tau & \leq \left\|u(t)\right\|^{2}_{\dot{H}^{\frac{1}{2}}}+|\beta|\left\|B(t)\right\|^{2}_{\dot{H}^{\frac{1}{2}}}+\left\|\Psi(t)\right\|^{2}_{\dot{H}^{\frac{1}{2}}}\\
&\quad+\frac{\mu}{2}\int^{t}_{0}\left(\left\|u(\tau)\right\|^{2}_{\dot{H}^{\frac{3}{2}}}+|\beta|\left\|B(\tau)\right\|^{2}_{\dot{H}^{\frac{3}{2}}}+\left\|\Psi(\tau)\right\|^{2}_{\dot{H}^{\frac{3}{2}}}\right)d\tau\\
&\leq \left\|u(t)\right\|^{2}_{L^{2}}+\left\|\omega(t)\right\|^{2}_{L^{2}} +\frac{\mu}{2}\int^{t}_{0}\left(\left\|\nabla u(\tau)\right\|^{2}_{L^{2}}+\left\|\nabla \omega(\tau)\right\|^{2}_{L^{2}}\right)d\tau \\
&\quad+|\beta|\left[ \left\|B(t)\right\|^{2}_{L^{2}} 
+\left\|J(t)\right\|^{2}_{L^{2}}
+\frac{\mu}{2}\int^{t}_{0}\left(\left\|\nabla B(\tau)\right\|^{2}_{L^{2}}+\left\|\nabla J(\tau)\right\|^{2}_{L^{2}}\right)d\tau\right] \\
&\quad+\norm{\Psi(t)}{\dot{H}^{\frac{1}{2}}}^2+\frac{\mu}{2}\int^{t}_{0}\left\|\Psi(\tau)\right\|^{2}_{\dot{H}^{\frac{3}{2}}} d\tau \\
&\leq C\left(\mathcal{E}_0+\frac{\mu^2}{C_2^2}\right).
\end{split}
\]

Hence, we deduce that 
\[
\int^{{T_{\text{max}}}}_0 \|u(\tau)\|_{{\text{BMO}}}^2+\|J(\tau)\|_{\text{{BMO}}}^2 \,d\tau\leq C\int^{{T_{\text{max}}}}_0 \|u(\tau)\|_{\dot{H}^{\frac{3}{2}}}^2+\|J(\tau)\|_{\dot{H}^{\frac{3}{2}}}^2 \,d\tau \leq C\left(\mathcal{E}_0+\frac{\mu^2}{C_2^2}\right),
\]
where BMO is  the space of functions of bounded mean oscillations \cite{John Nirenberg} and we use the continuous embedding $\dot{H}^{\frac{3}{2}}\hookrightarrow\text{BMO}$ in 3D \cite{Brezis Nirenberg} to the first inequality. This implies that $(u,B)$ can be extended over $t={T_{\text{max}}}$ from the blow-up criteria of \eqref{vH MHD} in \cite[Theorem 2]{Chae Lee}. It {is contradictory} to the definition of ${T_{\text{max}}}$ {and thus we conclude that}  ${T_{\text{max}}}=\infty$.

\section*{Acknowledgments}
H. Bae was supported by the National Research Foundation of Korea (NRF) grant funded by the Korea government (MSIT) (RS-2024-00341870). 

K. Kang was supported by the National Research Foundation of Korea(NRF) grant funded by the Korea government (MSIT) (RS-2024-00336346 and RS-2024-00406821). 

J. Shin was supported by the National Research Foundation of Korea(NRF) grant funded by the Korea government (MSIT) (RS-2024-00406821).

\section*{Data Availability}

Data sharing is not applicable to this article as no datasets were created or analyzed in this study.

\end{document}